\PassOptionsToPackage{noend}{algorithm2e}
\documentclass[final,12pt]{clear2025-arxiv} 

\usepackage{algorithm}
\usepackage{algorithmicx}
\usepackage{enumerate}

\usepackage{bbm}
\usepackage{pifont}
\usepackage{booktabs}
\usepackage{xcolor}
\definecolor{darkgreen}{rgb}{0,0.5,0}
\usepackage{hyperref}
\hypersetup{
    unicode=false,          
    colorlinks=true,        
    linkcolor=blue,          
    citecolor=darkgreen,    
    filecolor=magenta,      
    urlcolor=blue           
}
\usepackage[capitalize, nameinlink]{cleveref}

\usepackage{bm}
\usepackage{wasysym}
\usepackage{enumitem}

\usepackage{thmtools}
\usepackage{thm-restate}

\newtheorem{mylemma}[theorem]{Lemma}
\newtheorem{mycorollary}[theorem]{Corollary}

\newtheorem{mydefinition}[theorem]{Definition}
\newtheorem{myassumption}[theorem]{Assumption}

\newenvironment{myproof}[1][]{\medskip\noindent\textbf{Proof of #1:}\ }{\hfill$\blacksquare$}
\newenvironment{proofsketch}[1][]{\medskip\noindent\textbf{Proof sketch of #1:}\ }{\hfill$\blacksquare$}

\usepackage{tikz}
\usetikzlibrary{calc, graphs, graphs.standard, shapes, arrows, arrows.meta, positioning, decorations.pathreplacing, decorations.markings, decorations.pathmorphing, fit, matrix, patterns, shapes.misc, tikzmark}

\newcommand{\ba}{\mathbf{a}}
\newcommand{\bb}{\mathbf{b}}
\newcommand{\bc}{\mathbf{c}}

\newcommand{\bs}{\mathbf{s}}
\newcommand{\bu}{\mathbf{u}}

\newcommand{\bx}{\mathbf{x}}
\newcommand{\by}{\mathbf{y}}
\newcommand{\bz}{\mathbf{z}}
\newcommand{\bbeta}{\bm{\beta}}
\newcommand{\bgamma}{\bm{\gamma}}

\newcommand{\bA}{\mathbf{A}}
\newcommand{\bB}{\mathbf{B}}
\newcommand{\bC}{\mathbf{C}}
\newcommand{\bD}{\mathbf{D}}

\newcommand{\bP}{\mathbf{P}}
\newcommand{\bQ}{\mathbf{Q}}
\newcommand{\bV}{\mathbf{V}}
\newcommand{\bS}{\mathbf{S}}
\newcommand{\bU}{\mathbf{U}}

\newcommand{\bX}{\mathbf{X}}
\newcommand{\bY}{\mathbf{Y}}
\newcommand{\bZ}{\mathbf{Z}}
\newcommand{\bSigma}{\mathbf{\Sigma}}

\newcommand{\cE}{\mathcal{E}}
\newcommand{\cG}{\mathcal{G}}

\newcommand{\cO}{\mathcal{O}}

\newcommand{\ind}{\perp\!\!\!\!\perp}
\newcommand{\DO}{\texttt{do}}
\newcommand{\bbP}{\mathbb{P}}
\newcommand{\bbE}{\mathbb{E}}

\newcommand{\Bern}{\textnormal{Bern}}
\newcommand{\Bin}{\textnormal{Bin}}
\newcommand{\Pois}{\textnormal{Pois}}
\newcommand{\subG}{\textnormal{subG}}

\newcommand{\hatT}{\widehat{T}}
\newcommand{\hatbbP}{\widehat{\bbP}}

\newcommand{\ND}{\textnormal{ND}}
\newcommand{\De}{\textnormal{De}}
\newcommand{\Pa}{\textnormal{Pa}}

\newcommand{\An}{\textnormal{An}}
\newcommand{\Dis}{\textnormal{Dis}}

\renewcommand{\algorithmiccomment}[1]{\hfill$\rhd$ #1}

\newcommand{\ApproxCondInd}{\textnormal{\textsc{ApproxCondInd}}}
\newcommand{\AMBA}{\textnormal{\textsc{\textbf{A}pproximate\textbf{M}arkov\textbf{B}lanket\textbf{A}djustment}}}
\newcommand{\BAMBA}{\textnormal{\textsc{\textbf{B}eyond\textbf{A}pproximate\textbf{M}arkov\textbf{B}lanket\textbf{A}djustment}}}
\newcommand{\AMBAshort}{\textnormal{\textsc{Amba}}}
\newcommand{\BAMBAshort}{\textnormal{\textsc{Bamba}}}

\newcommand{\totaltol}{\lambda}

\newcommand{\bbR}{\mathbb{R}}
\newcommand{\kron}{\mathbbm{1}}
\newcommand{\wh}{\widehat}
\newcommand{\wt}{\widetilde}
\newcommand{\eps}{\varepsilon}
\newcommand{\xor}{\oplus}

\newcommand{\VminusXY}{\bV \setminus (\bX \cup \bY)}
\definecolor{colorone}{RGB}{0,0,150}
\definecolor{colortwo}{RGB}{150,0,0}

\newcommand{\estimationerror}{|\widehat{T}_{\bS,\bx,\by} - T_{\bS,\bx,\by}|}

\allowdisplaybreaks

\title[PAC estimation of valid adjustments]{Probably approximately correct high-dimensional causal effect
estimation given a valid adjustment set}
\usepackage{times}

\clearauthor{
\Name{Davin {Choo}}\thanks{Equal contribution} \Email{davin@u.nus.edu}\\
\addr National University of Singapore
\AND
\Name{Chandler {Squires}}\footnotemark[1]
\Email{csquires@mit.edu}\\
\addr Massachusetts Institute of Technology
\AND
\Name{Arnab {Bhattacharyya}}\\
\addr National University of Singapore
\AND
\Name{David {Sontag}}\\
\addr Massachusetts Institute of Technology
}

\begin{document}

\maketitle

\begin{abstract}%
Accurate estimates of causal effects play a key role in decision-making across applications such as healthcare, economics, and operations.
In the absence of randomized experiments, a common approach to estimating causal effects uses \textit{covariate adjustment}.
In this paper, we study covariate adjustment for discrete distributions from the PAC learning perspective, assuming knowledge of a valid adjustment set $\bZ$, which might be high-dimensional.
Our first main result PAC-bounds the estimation error of covariate adjustment by a term that is exponential in the size of the adjustment set; it is known that such a dependency is unavoidable even if one only aims to minimize the mean squared error.
Motivated by this result, we introduce the notion of an \emph{$\eps$-Markov blanket}, give bounds on the misspecification error of using such a set for covariate adjustment, and provide an algorithm for $\eps$-Markov blanket discovery; our second main result upper bounds the sample complexity of this algorithm.
Furthermore, we provide a misspecification error bound and a constraint-based algorithm that allow us to go beyond $\eps$-Markov blankets to even smaller adjustment sets.
Our third main result upper bounds the sample complexity of this algorithm, and our final result combines the first three into an overall PAC bound.
Altogether, our results highlight that one does not need to perfectly recover causal structure in order to ensure accurate estimates of causal effects.
\end{abstract}


\addtocontents{toc}{\protect\setcounter{tocdepth}{0}}

\section{Introduction}
\label{sec:introduction}

Let $\bbP(\bV)$ be an \emph{unknown} probability distribution over discrete random variables $\bV$.
Using i.i.d.\ samples from $\bbP(\bV)$, we wish to estimate the probability that $\bY \subset \bV$ equals $\by$, in the interventional distribution where $\bX \subset \bV$ is set to $\bx$.
This problem, called \emph{causal effect estimation}, can be formalized in either the Neyman-Rubin potential outcomes (PO) framework \citep{rubin1974estimating,splawa1990application,sekhon2009neyman} or in Pearl's graphical causality framework \citep{pearl2009causal}.
Depending on the framework, the desired estimand is written as $\bbP(\bY(\bx) = \by)$ or $\bbP_{\bx}(\by) = \bbP(\bY = \by \mid \DO(\bX = \bx))$, and has several important downstream applications such as estimating treatment effects.
In this work, we consider the problem from the viewpoint of distribution learning \citep{kearns1994learnability} under the Probably Approximately Correct (PAC) learning model \citep{valiant1984theory}.

\medskip
\noindent
\textbf{The PAC causal effect estimation (PAC-CEE) problem.}
Given (1) estimation tolerance $\totaltol > 0$, (2) failure tolerance $\delta > 0$, (3) sample access to a distribution $\bbP(\bV)$, and (4) an interventional query $\bbP_{\bx}(\by)$, output an estimate $\hatbbP_{\bx}(\by)$ such that 
$
\Pr \left( 
\left| \hatbbP_{\bx}(\by) - \bbP_{\bx}(\by) \right| \leq \totaltol
\right) 
\geq 1 - \delta
$.
\medskip

For this problem to be well-posed, one must be able to relate the observational distribution $\bbP(\bV)$ to the interventional distribution $\bbP_{\bx}(\by)$ via some \textit{identification formula}, i.e., $\bbP_{\bx}(\by)$ must be uniquely determined by $\bbP(\bV)$.
Here, we focus on a commonly studied identification formula that involve a set of variables $\bZ \subset \VminusXY$ such that $\bbP_{\bx}(\by) = T_{\bZ, \bx, \by}$, where
\begin{equation}
\label{eqn:covariate-adjustment}
T_{\bZ, \bx, \by}
:= \sum\nolimits_{\bz \in \bSigma_{\bZ}} \bbP(\bY = \by \mid \bZ = \bz, \bX = \bx) \cdot \bbP(\bZ = \bz)
= \sum\nolimits_{\bz} \bbP(\by \mid \bz, \bx) \cdot \bbP(\bz),
\end{equation}
with $\bSigma_{\bZ}$ denoting the alphabet of the variables $\bZ$.
For instance, in the PO framework, \cref{eqn:covariate-adjustment} holds under the assumptions of \textit{consistency} and \textit{conditional ignorability} of $\bX$ with respect to $\bZ$; see \cref{sec:appendix-potential-outcomes-adjustment} for a simple derivation.
Meanwhile, in the graphical framework, \cref{eqn:covariate-adjustment} can be shown to hold if $\bZ$ satisfies certain graphical criterion with respect to $\bX$ and $\bY$, such as the \textit{(generalized) backdoor criterion} or the \textit{(generalized) adjustment criteria} \citep{pearl1995causal, shpitser2010validity, maathuis2015generalized, perkovic2018complete}.
Following the latter viewpoint, we call $\bZ$ a \textit{valid adjustment set} for $\bbP_{\bx}(\by)$ if $\bZ$ satisfies $\bbP_{\bx}(\by) = T_{\bZ, \bx, \by}$ in \cref{eqn:covariate-adjustment}, but we emphasize that our results are framework-agnostic, i.e., they do not depend on how \cref{eqn:covariate-adjustment} is derived.

In particular, we establish our PAC guarantees by directly analyzing the sample complexity required to produce an estimate $\hatT_{\bZ, \bx, \by}$ of $T_{\bZ, \bx, \by}$.
A recent work of \citet{zeng2024causal} shows that $\Omega \left( \frac{1}{\lambda^2 \alpha_\bZ} + \frac{|\bSigma_{\bZ}|}{\lambda \alpha_\bZ} \right)$ samples are sufficient to ensure an expectation bound of $\bbE \left( |T_{\bZ,\bx,\by} - \hatT_{\bZ,\bx,\by}| \right) \leq \lambda$, where $\alpha_\bZ$ is a \textit{positivity} (a.k.a.\ \textit{overlap}) parameter that is common in causal effect estimation; we translate their actual stated bound into the form we describe here in \cref{sec:appendix-ZBHK24-derivation}.
\citet{zeng2024causal} also presented a minimax lower bound showing that linear dependency on $|\bSigma_{\bZ}|$ is unavoidable.
Since $|\bSigma_{\bZ}|$ grows exponentially with the size of $\bZ$ (e.g.\ when all variables are binary, we have $|\bSigma_{\bZ}| = 2^{|\bZ|}$), it is critical to use \textit{small} adjustment sets whenever possible.

In our work, given a valid adjustment set $\bZ \subseteq \bV$ as an initial input, we explore the possibility of searching for smaller adjustment sets with the objective of using less total samples than directly producing a $\lambda$-good estimate $\hatT_{\bZ,\bx,\by}$.
We are able to obtain lower sample complexities because of the adage from the property testing literature that ``testing can be cheaper than learning''.
In particular, we develop testing-based algorithms to find a candidate adjustment set $\bS \subseteq \bZ$, then estimate $\hatT_{\bS,\bx,\by}$ and bound its error from $\bbP_{\bx}(\by) = T_{\bZ,\bx,\by}$ via triangle inequality:
\[
\left| \bbP_{\bx}(\by) - \hatT_{\bS,\bx,\by} \right|
= \left| T_{\bZ,\bx,\by} - \hatT_{\bS,\bx,\by} \right|
\leq \left| T_{\bZ,\bx,\by} - T_{\bS,\bx,\by} \right| + \left| T_{\bS,\bx,\by} - \hatT_{\bS,\bx,\by} \right|
\leq \eps_1 + \eps_2
= \lambda
\]
The overall error bound ($\lambda$) is at most the sum of the misspecification bias error term ($\eps_1$) and the estimation error term ($\eps_2$).
There is an inherent tradeoff between these two sources of error: using $\bS \subseteq \bZ$ for adjustment might introduce misspecification bias (if $\bS$ is not a valid adjustment set), but this bias may dominated by a corresponding reduction in estimation error if $\bS$ is sufficiently smaller than $\bZ$.
While our approach to selecting $\bS$ is best appreciated through the lens of the graphical causality framework, it also applies in the PO setting, as we only rely on conditional independence tests using i.i.d.\ samples from $\bbP(\bV)$.

\subsection{Our main results}\label{sec:intro-main-results}


Our first main result extends the result of \citet{zeng2024causal} to the PAC setting by bounding the \textit{estimation error} $|T_{\bA,\bx,\by} - \hatT_{\bA,\bx,\by}|$ for arbitrary subsets $\bA \subseteq \VminusXY$, where $\hatT_{\bA,\bx,\by}$ is the estimate of $T_{\bA,\bx,\by}$ obtained using empirical sample estimates of $\bbP(\by \mid \ba,  \bx)$ and $\bbP(\ba)$ for all $\ba \in \Sigma_\bA$.
Throughout the paper, for any $\bA \subseteq \VminusXY$ arbitrary, we let 
\begin{equation}\label{eq:alpha-def}
    \alpha_{\bA} = \min\nolimits_{\ba \in \bSigma_{\bA}} \bbP(\bx \mid \ba)
\end{equation}

\begin{restatable}[Estimation error]{mytheorem}{estimationerror}
\label{thm:estimation-error}
Suppose we are given (1) estimation tolerance $\eps > 0$, (2) failure tolerance $\delta > 0$, (3) sample access to $\bbP(\bV)$, and (4) a subset $\bA \subseteq \VminusXY$.
Then, there is an algorithm that uses $\wt{\cO} \left( \left( \frac{|\bSigma_{\bA}|}{\eps \alpha_{\bA}} + \frac{1}{\eps^2 \alpha_{\bA}} + \frac{|\bSigma_{\bA}|}{\eps^2} \right) \cdot \log \frac{1}{\delta} \right)$ samples and produces an estimate $\hatT_{\bA, \bx, \by}$ such that
$
\Pr(|\hatT_{\bA, \bx, \by} - T_{\bA, \bx, \by}| \leq \eps) \geq 1 - \delta
$.
\end{restatable}

Note that, up to logarithmic factors and the additional $\frac{|\bSigma_{\bA}|}{\eps^2}$ factor, the sample complexity of the PAC bound matches the sample complexity of the expectation bound.
Here, we switched from $\lambda$ to $\eps$ and from $\bZ$ to $\bA$ to emphasize that the estimation error is only one part of our overall bound.
Surprisingly, although covariate adjustment is one of the simplest and most widely-used estimation techniques in causality, this result is (to the best of our knowledge) the first PAC bound on causal effect estimation for discrete variables.
In particular, previous works either focus on different estimands (under additional assumptions such as knowing a causal graph) or consider continuous variables and primarily provide only asymptotic results; we discuss related works in \cref{sec:related-work-feature-selection}.

Importantly, the sample complexity depends exponentially on $|\bA|$, hence it is crucial to keep $\bA$ small. 
As a paradigmatic example, consider the causal graph given in \cref{fig:intro-motivation}: instead of directly using $\bZ = \{A_1, \ldots, A_k, B\}$, there are two possible smaller subsets within $\bZ$ itself that satisfy the (generalized) backdoor adjustment criterion \citep{pearl1995causal, shpitser2010validity, maathuis2015generalized, perkovic2018complete}, and that therefore also serve as valid adjustment sets.

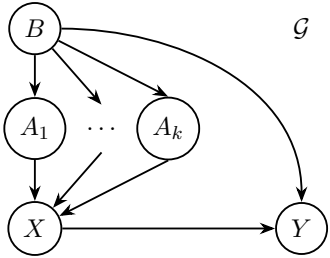
\begin{figure}[htb]
\centering
\resizebox{0.3\linewidth}{!}{
\begin{tikzpicture}
\node[] at (4,3) {$\cG$};
\node[draw, thick, minimum size=20pt, circle] at (0,0) (x) {$X$};
\node[draw, thick, minimum size=20pt, circle] at (4,0) (y) {$Y$};
\node[draw, thick, minimum size=20pt, circle] at (0,1.5) (a1) {$A_1$};
\node[minimum size=20pt] at (1,1.5) (dots) {$\ldots$};
\node[draw, thick, minimum size=20pt, circle] at (2,1.5) (ak) {$A_k$};
\node[draw, thick, minimum size=20pt, circle] at (0,3) (b) {$B$};
\draw[thick, -Stealth] (x) -- (y);
\draw[thick, -Stealth] (b) to[out=0, in=90] (y);
\draw[thick, -Stealth] (b) -- (a1);
\draw[thick, -Stealth] (b) -- (dots.north);
\draw[thick, -Stealth] (b) -- (ak.north);
\draw[thick, -Stealth] (a1) -- (x);
\draw[thick, -Stealth] (dots.south) -- (x);
\draw[thick, -Stealth] (ak.south) -- (x);
\end{tikzpicture}
}
\caption{Consider the graphical causality framework and $\bbP_{x}(y)$ in $\cG$ with $\bZ = \{ A_1, \ldots, A_k, B \}$ as a valid adjustment set.
Both the parental set $\Pa(X) = \{ A_1, \ldots, A_k \}$ and the singleton set $\{B\}$ satisfy the backdoor adjustment criterion \citep{pearl1995causal} and are also valid.}
\label{fig:intro-motivation}
\end{figure}

As a first approach for obtaining smaller adjustment sets, we consider Markov blankets of the treatments $\bX$.
In particular, we say that $\bS \subseteq \bZ$ is a \textit{Markov blanket} of $\bX$ with respect to $\bZ$ when $\bX \ind \bZ \setminus \bS \mid \bS$, and one can show that $T_{\bS, \bx, \by}$ = $T_{\bZ, \bx, \by}$ when this conditional independence holds.
In the example given in \cref{fig:intro-motivation}, the parental set $\Pa(X) = \{ A_1, \ldots, A_k \}$ satisfies this condition: $X \ind \{B\} \mid \Pa(X)$.
To adapt this notion in the finite sample setting, we consider an approximate version of conditional independence and define a parameter $\Delta_{\bX \ind \bZ \setminus \bS \mid \bS}$ that measures how much the conditional independence condition is violated in $\bbP(\bV)$.
When $\Delta_{\bX \ind \bZ \setminus \bS \mid \bS} \leq \eps$, we say that $\bS$ is an \textit{$\eps$-Markov blanket} of $\bX$ with respect to $\bZ$.

\begin{mydefinition}[Approximate conditional independence]
\label{def:approx-cond-ind}
For disjoint sets $\bA, \bB, \bC \subseteq \bV$, we define $\Delta_{\bA \ind \bB \mid \bC} = \sum_{\ba, \bb, \bc} \bbP(\bc) \cdot | \bbP(\ba, \bb \mid \bc) - \bbP(\ba \mid \bc) \cdot \bbP(\bb \mid \bc) |$.
If $\Delta_{\bA \ind \bB \mid \bC} \leq \eps$, we write $\bA \ind_{\eps} \bB \mid \bC$.
\end{mydefinition}

\begin{mydefinition}[(Approximate) Markov blanket]
\label{def:markov-blanket}
Consider an arbitrary subset $\bA \subseteq \VminusXY$.
A subset $\bS \subseteq \bA$ is called a \emph{Markov blanket} of $\bX$ with respect to $\bA$ if $\bX \ind \bA \setminus \bS \mid \bS$ and an \emph{$\eps$-Markov blanket} if $\bX \ind_\eps \bA \setminus \bS \mid \bS$.
\end{mydefinition}

We show that $| T_{\bA, \bx, \by} - T_{\bS, \bx, \by} | \leq \frac{\eps}{\alpha_\bS}$ whenever $\bS$ is an $\eps$-Markov blanket of $\bA$.
In particular, if $\bA = \bZ$ is a valid adjustment set, then this bound applies to the misspecification bias mentioned above.
Our next result bounds the sample complexity for discovering an $\eps$-Markov blanket.

\begin{restatable}[Approximate Markov blanket discovery]{mytheorem}{approximatemarkovblanketdiscovery}
\label{thm:AMBA}
Suppose we are given (1) $\eps > 0$, (2) $\delta > 0$, (3) sample access to a distribution $\bbP(\bV)$, and (4) an arbitrary subset $\bA \subseteq \VminusXY$.
Suppose that there is a Markov blanket of $\bX$ with respect to $\bA$ with $k$ variables.
Then, there is an algorithm that uses $\wt{\cO} \left( \frac{|\bS|}{\eps^2} \cdot \sqrt{|\bSigma_{\bX}| \cdot |\bSigma_{\bA}|} \cdot \log \frac{1}{\delta} \right)$ samples and produces a subset $\bS \subseteq \bA$ such that $|\bS| \leq k$, $\Pr \left( \Delta_{\bX \ind \bA \setminus \bS \mid \bS} > \eps \right) \geq 1 - \delta$, and $\Pr \left( | T_{\bS, \bx, \by} - T_{\bA, \bx, \by} | \leq \frac{\eps}{\alpha_{\bS}} \right) \geq 1 - \delta$.
\end{restatable}

Now, suppose that \cref{thm:AMBA} outputs $\bS \subseteq \bZ$ when given a valid adjustment set $\bZ$.
While $|\bS|$ may be smaller than $|\bZ|$, it may still be much larger than the smallest valid adjustment set for $\bbP_\bx(\by)$.
For example, we see that $|\bZ| = k+1 > k = |\Pa(X)| \gg |\{B\}| = 1$ in \cref{fig:intro-motivation} where $\bZ$, $\Pa(X)$, and $\{B\}$ are all valid adjustment sets.
Our next result aims to find an adjustment set $\bS' \subseteq \bZ$ of \emph{minimal size} given a valid adjustment set $\bZ$ and an $\eps$-Markov blanket $\bS \subseteq \bZ$ of it.
To this end, we introduce the more general concept of a \emph{screening set} of an arbitrary subset $\bA \subseteq \VminusXY$.
\begin{mydefinition}[(Approximate) Screening set]
Let $\bA \subseteq \VminusXY$ and $\bS \subseteq \bA$.
A subset $\bB \subseteq \bA$ is called a \emph{screening set} for $(\bS, \bA, \bX, \bY)$ if $\bY \ind \bS \setminus \bB \mid \bX \cup \bB$ and $\bX \ind \bB \setminus \bS \mid \bS$.
Meanwhile, the subset $\bB$ is called an \emph{$\eps$-screening set} for $(\bS, \bA, \bX, \bY)$ if $\bY \ind_\eps \bS \setminus \bB \mid \bX \cup \bB$ and $\bX \ind_\eps \bB \setminus \bS \mid \bS$.
\end{mydefinition}

As a technical side note (see the exposition after \cref{lem:minimal-soundness}), given an adjustment set $\bS$, the screening set condition for $(\bS, \bA, \bX, \bY)$ is sound for $\bB$ to be a valid adjustment set, but it is incomplete in general, in that sense that there may exist valid adjustment sets that do not satisfy the screening set condition.
In the worst case, our algorithm in \cref{thm:BAMBA} will output $\bS' = \bS$.

\begin{restatable}[Beyond approximate Markov blankets]{mytheorem}{minimalgivenapproximatemarkovblanketdiscovery}
\label{thm:BAMBA}
Suppose we are given (1) $\eps > 0$, (2) $\delta > 0$, (3) sample access to $\bbP(\bV)$, (4) an arbitrary subset $\bA \subseteq \VminusXY$, and (5) an $\eps$-Markov blanket $\bS \subseteq \bA$.
Suppose there is a screening set $\bB$ for $(\bS, \bA, \bX, \bY)$ such that $|\bB| = k'$ and $|\Sigma_{\bB}| \leq |\Sigma_\bS|$.
There is an algorithm that uses $\wt{\cO} \left( \frac{|\bS'|}{\eps^2} \cdot \sqrt{|\bSigma_{\bX}| \cdot |\bSigma_{\bY}| \cdot |\bSigma_{\bA}|} \cdot \log \frac{1}{\delta} \right)$ samples and produces a subset $\bS' \subseteq \bA$ such that $|\bS'| \leq k'$, $|\bSigma_{\bS'}| \leq |\bSigma_{\bS}|$ and $\Pr \left( | T_{\bS', \bx, \by} - T_{\bA, \bx, \by} | \leq \frac{2 \eps}{\alpha_{\bS}} \right) \geq 1 - \delta$.
\end{restatable}

As we shall see, unlike many existing causal discovery methods, e.g.\ the PC algorithm \citep{spirtes2000causation}, which perform a sequence of dependent conditional independence checks, our algorithms for \cref{thm:AMBA} and \cref{thm:BAMBA} use a \emph{non-dependent} collection of conditional independence tests, allowing us to avoid error propagation and control the sample complexity of our procedures.

Finally, one can combine the PAC bound results above to yield an overall PAC bound guarantee for solving the causal effect estimation problem as follows.
Since \cref{lem:misspecification-error} tells us that $T_{\bS, \bx, \by}$ and $T_{\bZ, \bx, \by}$ are close whenever $\bS$ is an $\eps$-Markov blanket of $\bX$ with respect to $\bZ$, we can employ the algorithm in \cref{thm:AMBA} to find a subset $\bS \subseteq \bZ$ such that $T_{\bS, \bx, \by} \approx T_{\bZ, \bx, \by}$.
Using the $\eps$-Markov blanket $\bS$, we can further use the algorithm in \cref{thm:BAMBA} to find a subset $\bS' \subseteq \bZ$ such that $T_{\bS', \bx, \by} \approx T_{\bZ, \bx, \by}$.
Depending on whether $|\bSigma_{\bS}|$ or $|\bSigma_{\bS'}|$ is smaller, we can employ \cref{thm:estimation-error} to obtain an estimate $\hatT_{\bS, \bx, \by}$ or $\hatT_{\bS', \bx, \by}$, and use that as an estimate for $T_{\bZ, \bx, \by} = \bbP_{\bx}(\by)$.

In practical situations where one is given a fixed number of samples, we can re-express the results of \cref{thm:estimation-error}, \cref{thm:AMBA} and \cref{thm:BAMBA} in terms of an error upper bound.
Then, one can derive a condition under which a combined approach based on above results estimates $\hatbbP_{\bx}(\by)$ via $\hatT_{\bS',\bx,\by}$, for some $\bS' \subseteq \bZ$, and provably achieves a smaller asymptotic error than directly estimating $\hatbbP_{\bx}(\by)$ via $\hatT_{\bZ,\bx,\by}$.
The condition relies on the positivity of $\alpha_{\bS}$ for subsets $\bS, \bS' \subseteq \bZ$ which are unknown a priori.
However, if one is willing to make lower bound assumptions on these $\alpha$ values, possibly due to background knowledge, then one can obtain a result in the same vein as \cref{thm:pac-causal-effect-estimation-special-case}.

\begin{restatable}[PAC causal effect estimation with positivity]{mytheorem}{paccausaleffectestimationspecialcase}
\label{thm:pac-causal-effect-estimation-special-case}
Suppose we are given (1) $\eps > 0$, (2) $\delta > 0$, (3) $n$ i.i.d.\ samples from $\bbP(\bV)$, (4) an interventional query $\bbP_{\bx}(\by)$, (5) a valid adjustment set $\bZ \subseteq \VminusXY$, and (6) guaranteed that $\alpha_{\bS} \geq \alpha \in (0,1)$ for any $\bS \subseteq \bZ$.
Then, there is an algorithm that outputs a subset $\bS^* \subseteq \bZ$ and an estimate $\hatbbP_{\bx}(\by) = \hatT_{\bS^*,\bx,\by}$ such that $\Pr \left( \left| \hatbbP_{\bx}(\by) - \bbP_{\bx}(\by) \right| \leq \eps \right) \geq 1 - \delta$ for some error term
\[
\eps \in \wt{\cO} \left(
\frac{1}{n} \cdot \frac{|\bSigma_{\bS^*}|}{\alpha} + \frac{1}{\sqrt{n}} \cdot \left( \frac{\sqrt{|\bZ|} \cdot \left( |\bSigma_{\bX}| \cdot |\bSigma_{\bY}| \cdot |\bSigma_{\bZ}| \right)^{\frac{1}{4}}}{\alpha} + \frac{1}{\sqrt{\alpha}} + \sqrt{|\bSigma_{\bS^*}|}
\right) \right).
\]
Moreover, if there exists a Markov blanket $\bS$ of $\bX$ such that $|\bS| \cdot \sqrt{\frac{|\bSigma_{\bX}|}{|\bSigma_{\bZ}|}} < \max \left\{ \frac{|\bSigma_{\bZ}|}{n}, \frac{\alpha_{\bS}}{|\bSigma_{\bZ}|}, \alpha_{\bS}^2 \right\}$, then $|\bS^*| \leq k$.
\end{restatable}

\subsection{Overview of technical results}

We now review some additional technical results used to establish our main results, though we note that these results may be of independent interest for future work.
While our notation and language is closer to Pearl's graphical causal modeling framework \citep{pearl2009causal}, all of our results are compatible with both the PO and graphical frameworks as long as \cref{eqn:covariate-adjustment} holds for the given $\bZ$.
This is because our analysis is purely probabilistic in nature, with the causal interpretation always going back to assuming that $\bbP_{\bx}(\by) = T_{\bZ, \bx, \by}$ as a starting point.

The sample complexity bound of \cref{thm:estimation-error} heavily relies on a common technique in the property testing literature known as Poissonization, e.g.\ see \cite[Section 4.3]{valiant2008testing}, \cite[Appendix D.3]{canonne2020survey}, and \cite[Appendix C]{canonne2022topics}.
The high level idea is that instead of drawing $n$ i.i.d.\ samples, we will draw $N_{\Pois} \sim \Pois(n)$ i.i.d\ samples, where $N_{\Pois}$ is a random Poisson variable, so that the random count for each realized value will be independent.
\cref{sec:poissonization} describes the Poissonization sampling technique in further detail.

In our error analyses in \cref{thm:AMBA} and \cref{thm:BAMBA}, we manipulate approximate conditional independence terms $\Delta_{\bA \ind \bB \mid \bC}$ (from \cref{def:approx-cond-ind}) and adjustment terms $T_{\bA, \bx, \by}$ (from \cref{eqn:covariate-adjustment}) for various subsets $\bA, \bB, \bC \subseteq \bV$.
By standard probability manipulations, one can easily obtain the following alternative representation of $\Delta_{\bA \ind \bB \mid \bC}$ for arbitrary disjoint subsets $\bA, \bB, \bC \subseteq \bV$.

\begin{equation}
\label{eq:approx-cond-ind-alternative}
\begin{aligned}
\Delta_{\bA \ind \bB \mid \bC}
&= \sum\nolimits_{\ba, \bb, \bc} \bbP(\bc) \cdot | \bbP(\ba, \bb \mid \bc) - \bbP(\ba \mid \bc) \cdot \bbP(\bb \mid \bc) |
\\
&= \sum\nolimits_{\ba, \bb, \bc} \bbP(\ba, \bc) \cdot | \bbP(\bb \mid \ba, \bc) - \bbP(\bb \mid \bc) |
\leq \eps
\end{aligned}
\end{equation}

Meanwhile, for any arbitrary disjoint subsets $\bA, \bB \subseteq \VminusXY$, $T_{\bA, \bx, \by}$ can be re-expressed in multiple ways (depending on the desired analytical use case) using law of total probability as

\begin{equation}
\label{eq:T-alternative}
T_{\bA, \bx, \by}
= \sum\nolimits_{\ba} \bbP(\by \mid \ba, \bx) \cdot \bbP(\ba)
= \sum\nolimits_{\ba, \bb} \bbP(\by \mid \ba, \bx) \cdot \bbP(\ba) \cdot \bbP(\bb \mid \ba)
\end{equation}

The correctness of \cref{thm:AMBA} follows from the following result that $T_{\bS, \bx, \by}$ and $T_{\bA, \bx, \by}$ are close whenever $\bS$ is an $\eps$-Markov blanket of $\bS$ with respect to $\bA$, and that there is a sample efficient way to obtain such an $\eps$-Markov blanket.

\begin{restatable}[Misspecification error]{mylemma}{misspecificationerror}
\label{lem:misspecification-error}
If $\bS \subseteq \bA \subseteq \VminusXY$ such that $\bX \ind_\eps \bA \setminus \bS \mid \bS$, then $|T_{\bS,\bx,\by} - T_{\bA,\bx,\by}| \leq \frac{\eps}{\alpha_{\bS}}$.
\end{restatable}

We additionally complement \cref{lem:misspecification-error} with a hardness result of \cref{lem:misspecification-error-lower-bound}.

\begin{restatable}[Misspecification error lower bound]{mylemma}{misspecificationerrorlowerbound}
\label{lem:misspecification-error-lower-bound}
Let $0 \leq \sqrt{\eps} \leq \alpha \leq 1/2$.
There exists $\bbP(\bV)$ such that (i) $\bZ$ is a valid adjustment set, (ii) $\bS \subset \bZ$ satisfies $\bX \ind_\eps \bZ \setminus \bS \mid \bS$, (iii) $\alpha_{\bS} \geq \alpha$, and (iv) $|T_{\bS,\bx,\by} - T_{\bZ,\bx,\by}| \geq \frac{\eps}{16 \alpha}$.
\end{restatable}

Similar in spirit to \cref{thm:AMBA}, the correctness of \cref{thm:BAMBA} relies on the relating $T_{\bS', \bx, \by}$ and $T_{\bS, \bx, \by}$ via some conditional independence relations.
Given an $\eps$-Markov blanket $\bS \subseteq \bA$, we search for a minimal-sized screening set for $(\bS, \bX, \bY)$.
This is a sound approach because of \cref{lem:minimal-soundness}.

\begin{restatable}[Adjustment soundness]{mylemma}{minimalsoundness}
\label{lem:minimal-soundness}
Let $\bA \subseteq \VminusXY$ be an arbitrary subset and $\bS \subseteq \bA$.
If $\bS'$ is a screening set for $(\bS, \bX, \bY)$, then $T_{\bS', \bx, \by} = T_{\bS, \bx, \by}$.
\end{restatable}

While this approach is sound, it may not discover the smallest possible subset satisfying $T_{\bS', \bx, \by} = T_{\bS, \bx, \by}$ for any choice of $\bS \subseteq \bA$.
Nevertheless, there exists special scenarios in which this approach is also complete.
We give a full statement of such a graphical condition and a completeness proof in \cref{sec:appendix-completness}; here, we provide some brief intuition for the curious reader.
In particular, the necessity of the $\bY \ind \bS \setminus \bS' \mid \bX \cup \bS'$ condition can be appreciated by setting $\bA = \{A_1, \ldots, A_k, B\}$, $\bS = \{A_1, \ldots, A_k\}$, and $\bS' = \{B\}$ in \cref{fig:intro-motivation}.
In this setup, we see that $\bS$ is a valid backdoor adjustment set.
Observe that conditioning on $X$ blocks any paths from $\bS$ to $Y$ that has a causal path from $X$ to $Y$ as a subpath.
So, $\bY \ind \bS \setminus \bS' \mid \bX \cup \bS'$ will imply that $\bS'$ also blocks non-causal $X$ to $Y$ paths since any such path passing through $\bS \setminus \bS'$ has to pass through $\bS'$ to reach $Y$.


Finally, there is nothing technically special about \cref{thm:pac-causal-effect-estimation-special-case} besides simply combining the results \cref{thm:estimation-error}, \cref{thm:AMBA}, and \cref{thm:BAMBA} in a straightforward fashion.

\subsection{Organization of the paper}

The remainder of the paper is devoted to establishing and providing intuition behind our results; related work, additional derivations, and full proofs are deferred to the appendix.
In \cref{sec:prelims}, we establish our notation and describe necessary preliminaries.
In \cref{sec:sample-complexity,sec:AMBA,sec:BAMBA}, we prove our main results and their associated technical results.
In particular, we prove \cref{thm:estimation-error} in \cref{sec:sample-complexity} and the others (\cref{thm:AMBA}, \cref{thm:BAMBA}, \cref{lem:misspecification-error}, \cref{lem:misspecification-error-lower-bound}, and \cref{lem:minimal-soundness}) in \cref{sec:markov-blanket}.
We conclude with a summary and a discussion of open problems in \cref{sec:conclusion}.

\section{Preliminaries}
\label{sec:prelims}

\textbf{Notation.}
We use capital letters for random variables and lowercase letters for the realizations, e.g.\ $X = x$, $Y = y$, etc.
We use bold letters for sets of variables and write $\bbP(\bA = \ba)$ as $\bbP(\ba)$ as shorthand.
We denote the alphabet of the variable $V$ as $\bSigma_V$, and extend this to sets by letting $\bSigma_\bA = \bSigma_{V_1} \times \ldots \times \bSigma_{V_k}$, where $\bA = \{ V_1, \ldots, V_k \}$ and $\times$ denotes the Cartesian product.
To lighten notation, summations are always taken over the entire alphabet of the index, i.e., $\sum_\ba f(\ba)$ denotes $\sum_{\ba \in \Sigma_\bA} f(\ba)$.
We employ the standard asymptotic notations $\cO(\cdot)$, $\Omega(\cdot)$ $\Theta(\cdot)$, and $\wt{\cO}(\cdot)$.

Throughout this work, we will denote $\bX$ as the intervened treatment variables and $\bY$ as the outcome variables of interest.
For some $\bx \in \bSigma_\bX$ and $\by \in \bSigma$, our goal is to estimate $\bbP_\bx(\by)$, which denotes the probability that $\bY$ takes on the value $\by$ if we intervene to set $\bX$ equal to $\bx$.

\noindent\textbf{A short note on valid adjustment sets.}
In this work, we take as our starting point knowledge of some $\bZ \subset \VminusXY$ that is a valid adjustment set for $\bbP_\bx(\by)$, i.e., we assume that \cref{eqn:covariate-adjustment} holds for some known set $\bZ \subset \VminusXY$.
Instead of starting directly from this point, one may prefer to derive \cref{eqn:covariate-adjustment} from more foundational assumptions.
%
For example, in the potential outcomes (PO) framework, $\bbP_\bx(\by)$ is usually written as $\bbP(\bY(\bx) = \by)$, where $\bY(\bx)$ is a random variable denoting the potential outcome under an intervention that sets $\bX$ to $\bx$.
Then, \cref{eqn:covariate-adjustment} is implied under the standard consistency assumption and conditional ignorability of $\bX$ with respect to $\bZ$; see \cref{lem:Z-is-valid-under-SUTVA-and-conditional-ignorability}.
Alternatively, \cref{eqn:covariate-adjustment} can be derived in the graphical causality framework, which relates the distributions $\bbP(\bV)$ and $\bbP_\bx(\bY)$ to a (possibly unknown) \textit{causal graph} $\cG$ over the random variables $\bV$.
Typically, on might assume that $\cG$ an \emph{acyclic directed mixed graph} (ADMG), which can contain both directed edges (of the form $V_1 \to V_2$) and bidirected edges (of the form $V_1 \leftrightarrow V_2$), but no directed cycles; if there are also no bidirected edges, then $\cG$ is called a \textit{directed acyclic graph} (DAG).
%
%
%
An ADMG (resp.\ DAGs) $\cG$ defines a three-way relation between subsets $\bA, \bB, \bC \subseteq \bV$ called \textit{m-separation} (resp.\ \textit{d-separation}), and the distributions $\bbP(\bV)$ and $\bbP_\bx(\by)$ are assumed to be related via m-separation in $\cG$ and related graphs.
Then, \cref{eqn:covariate-adjustment} can be derived from these assumptions and graphical conditions on $\bZ$, see \cref{sec:related-work} for examples of such conditions.


\subsection{Poissonization}\label{sec:poissonization}

We now describe a common technique known as Poissonization; e.g.\ see \cite{valiant2008testing,canonne2020survey,canonne2022topics}.
%
When drawing $n$ i.i.d.\ samples from an underlying distribution $\bbP(X)$ over a domain $\bSigma_X = \{1, \ldots, k\}$, the vector of counts $(N_1, \ldots, N_k)$ follows a multinomial distribution with parameters $n$ and $(\bbP(X=1), \ldots, \bbP(X=k))$, where each random variable $N_i$ is the number of times we observe $i \in [k]$ amongst the $n = N_1 + \ldots + N_k$ drawn samples.
Oftentimes, in analysis, we would like that the random variables $N_1, \ldots, N_k$ are independent; unfortunately, this is false in the standard sampling setting since $N_1, \ldots, N_k$ are negatively correlated.

Instead of directly drawing $n$ i.i.d.\ samples, the idea behind Poissonization is to modify the sampling process by first sampling a Poisson number $N_{\Pois} \sim \Pois(n)$ with mean $n$ and then drawing $N_{\Pois}$ i.i.d\ samples.
By standard Poisson concentration bounds (e.g., \cref{lem:poisson-concentration} below), $N_\Pois$ is of order $\cO(n)$ with high probability; thus, PAC bounds for the Poissonized setting are interchangeable with those in the standard setting up to constant factors (see e.g., Lemmas C.1 and C.2 in \cite{canonne2022topics}).
In the Poissonized setting, the resulting count vector has a few desirable properties.

\begin{mylemma}[Appendix C of \cite{canonne2022topics}]
\label{lem:poissonization}
Let $(N_1, \ldots, N_k)$ be the sample counts in the Poissonized sampling process such that $N_1 + \ldots + N_k = N_{\Pois} \sim \Pois(n)$.
The following hold:\\
(a) The random count variables $N_1, \ldots, N_k$ are mutually independent.\\
(b) For each $i \in [k]$, we have $N_i \sim \Pois(n \cdot \bbP(X = i))$.\\
(c) For each $i \in [k]$ and natural number $n'$, we have $(N_i \mid N_{\Pois} = n') \sim \Bin(n', \bbP(X = i))$.
\end{mylemma}

\subsection{Concentration bounds}

The first two terms in \cref{thm:estimation-error} come from splitting the values of $\bA$ into two sets according to some cutoff parameter $\tau$ that we later optimize.
In \cref{eq:J-event}, we define an event $\cE^{J}_{\geq \tau}$; the $\wt{\cO} \left( \frac{|\bSigma_\bA|}{\eps \alpha_\bA} \right)$ term captures the number of samples needed to ensure the $\cE^{J}_{\geq \tau}$ holds with high probability, which we compute using a standard Poisson concentration bound (\cref{lem:poisson-concentration}).
Next, we define a random variable $J_{\geq \tau}$ (\cref{eq:J-geq-tau-defn}); the $\wt{\cO} \left( \frac{1}{\eps^2 \alpha_\bA} \right)$ term comes from a concentration bound on this term.
In particular, the bulk of our analysis is devoted to showing that $J_{\geq \tau}$ is sub-Gaussian conditioned on $\cE^{J}_{\geq \tau}$, for this result, we require Lemmas \ref{lem:subgaussian-sum}, \ref{lem:subgaussian-take-max}, and \ref{lem:hoeffding} below.

\begin{mylemma}[Poisson concentration; e.g.\ see Theorem A.8 in \cite{canonne2022topics}]
\label{lem:poisson-concentration}
Let $N \sim \Pois(n)$ be a Poisson random variable with parameter $n$.
Then, for any $0 < t < n$, we have $\Pr(N \leq n - t) \leq \exp \left( -\frac{t^2}{2(n+t)} \right)$.
In particular, setting $t = n/2$, we have $\Pr(N \leq n/2) \leq \exp \left( -\frac{n}{12} \right)$.
\end{mylemma}


\begin{mydefinition}[Sub-Gaussian distribution; e.g.\ see Section 1.2 of \cite{rigollet2023high}]
\label{def:subgaussian}
A random variable $X$ is \emph{sub-Gaussian} with parameter $\sigma^2$ if $\bbE(X) = 0$ and $\bbE \left( e^{ \lambda X } \right) \leq \exp \left( \frac{\lambda^2 \sigma^2}{2} \right)$ for all $\lambda \in \mathbb{R}$.
If $X \sim \subG(\sigma^2)$, it is known that $\Pr(|X| \geq t) \leq 2 \exp(-t^2 / \sigma^2)$ for any $t \geq 0$.
\end{mydefinition}

\begin{mylemma}[Sub-Gaussian additivity; e.g.\ see Corollary 1.7 of \cite{rigollet2023high}]
\label{lem:subgaussian-sum}
For $i \in [k]$, let $X_i \sim \subG(\sigma_i^2)$ be an independent sub-Gaussian random variable with parameter $\sigma_i^2$.
Then, for any set of real coefficients $a_1, \ldots, a_k \in \bbR$, we have $\left( \sum_{i=1}^k a_i X_i \right) \sim \subG(\sum_{i=1}^k a_i^2 \sigma_i^2)$.
\end{mylemma}

\begin{mylemma}[See \cref{sec:appendix-subgaussian-take-max}]
\label{lem:subgaussian-take-max}
Let $X$ and $Y$ be discrete random variables.
If $(X \mid Y = y) \sim \subG(\sigma_{y}^2)$ for every $y \in \bSigma_{Y}$, then $X \sim \subG(\max_{y \in \bSigma_{Y}} \sigma_{y}^2)$.
\end{mylemma}

\begin{mylemma}[Hoeffding's lemma; \cite{hoeffding1994probability}]
\label{lem:hoeffding}
Let $X$ be any real-valued random variable in the range $[a,b]$.
Then, for any $\lambda \in \mathbb{R}$, we have $\bbE \left( e^{\lambda (X - \bbE(X)} \right) \leq \exp \left( \frac{\lambda^2 (b-a)^2}{8} \right)$.
\end{mylemma}

\subsection{Known sample complexity results for discrete distributions}

The third term in \cref{thm:estimation-error} comes from reducing one subtask in our analysis to the problem of estimating $\bbP(\bA)$ well in TV distance, for some subset of variables $\bA \subseteq \bV$.
In the distribution testing literature, this task is well-known to require $\wt{\Theta} \left( \frac{|\bSigma_{\bA}|}{\eps^2} \right)$ i.i.d.\ samples.

\begin{mylemma}[Estimating well in TV; e.g.\ see \cite{canonne2020short}]
\label{lem:TV-estimation-guarantees}
Given tolerance parameters $\eps, \delta > 0$ and sample access to a distribution $\bbP(\bA)$, the empirical distribution $\wh{\bbP}(\bA)$ constructed from $\cO \left( \frac{|\bSigma_\bA| + \log \frac{1}{\delta}}{\eps^2} \right)$ i.i.d.\ samples has the property that $\Pr \left( \sum_{\ba \in \bSigma_{\bA}} | \bbP(\ba) - \wh{\bbP}(\ba) | \leq \eps \right) \geq 1 - \delta$.
\end{mylemma}

Meanwhile, for our proofs of Theorems \ref{thm:AMBA} and \ref{thm:BAMBA}, several methods have been developed which satisfy the requirements of the $\eps$-approximate conditional independence tester for \cref{def:approx-cond-ind}.
In this work, we call our $\eps$-approximate conditional independence tester \ApproxCondInd.
Assuming that $\eps^{-1}$ is sufficiently large\footnote{For instance, $\frac{1}{\eps} > |\bSigma_{\bC}|^{\frac{1}{4}} \cdot \left(\max\{ |\bSigma_{\bA}|, |\bSigma_{\bB}|, |\bSigma_{\bC}| \}\right)^{\frac{1}{4}}$ would suffice.} compared to $|\bSigma_{\bA}|$, $|\bSigma_{\bB}|$, and $|\bSigma_{\bC}|$, \cite{canonne2018testing} proposes a test based on total variation distance that uses $\wt{\cO} \left( \frac{1}{\eps^2} \cdot \sqrt{|\bSigma_{\bA}| \cdot |\bSigma_{\bB}| \cdot |\bSigma_{\bC}|} \right)$ samples from $\bbP$; see their Theorem 1.3 and Lemma 2.2.
There is also a simpler test based on the empirical mutual information, proposed by \cite{bhattacharyya2021near}, that uses $\wt{\cO} \left( \frac{1}{\eps^2} \cdot |\bSigma_{\bA}| \cdot |\bSigma_{\bB}| \cdot |\bSigma_{\bC}| \right)$ samples from $\bbP$, though we use the former to obtain optimal dependence on the alphabet sizes.

\begin{mylemma}[Using \cite{canonne2018testing} for \ApproxCondInd]
\label{lem:approxcondind-guarantees}
Given $\eps, \delta > 0$ and sample access to distribution $\bbP(\bV)$, \ApproxCondInd\ uses $\wt{\cO} \left( \frac{1}{\eps^2} \cdot \sqrt{|\bSigma_{\bA}| \cdot |\bSigma_{\bB}| \cdot |\bSigma_{\bC}|} \cdot \log \frac{1}{\delta} \right)$ samples and correctly determines whether $\Delta_{\bA \ind \bB \mid \bC} = 0$ (outputs YES) or $\Delta_{\bA \ind \bB \mid \bC} > \eps$ (outputs NO) with probability at least $1 - \delta$, for any disjoint sets $\bA, \bB, \bC \subseteq \bV$.
\end{mylemma}
Note that when $0 < \Delta_{\bA \ind \bB \mid \bC} \leq \eps$, \ApproxCondInd\ is allowed to output arbitrarily.
In particular, when $\ApproxCondInd$ outputs YES, then we have $\Pr(\Delta_{\bA \ind \bB \mid \bC} \leq \eps) \geq 1 - \delta$.

\section{Sample complexity for empirical estimation}
\label{sec:sample-complexity}

In this section, we prove \cref{thm:estimation-error}, our upper bound on the sample complexity of estimating $T_{\bA, \bx, \by}$ given any $\bA \subseteq \bV$.
For analysis purposes, we will use the Poissonization sampling process (\cref{sec:poissonization}) so that we invoke \cref{lem:poissonization} to obtain PAC style bounds.

\begin{proofsketch}[\cref{thm:estimation-error}]
By definition, we have
\begin{align*}
T_{\bA,\bx,\by} - \hatT_{\bA,\bx,\by}
&= \sum_\ba \left( \bbP(\ba) \cdot \bbP(\by \mid \bx, \ba) - \frac{N_\ba}{N_\Pois} \cdot \frac{N_{\by,\bx,\ba}}{N_{\bx,\ba}} \right)\\
&= \sum_\ba \bbP(\ba) \cdot \left( \bbP(\by \mid \bx, \ba) - \frac{N_{\by,\bx,\ba}}{N_{\bx,\ba}} \right) + \sum_\ba \left( \bbP(\ba) - \frac{N_\ba}{N_\Pois} \right) \cdot \frac{N_{\by,\bx,\ba}}{N_{\bx,\ba}}
\end{align*}
where $N_\ba$, $N_\Pois$, $N_{\by,\bx,\ba}$, and $N_{\bx,\ba}$ are random Poisson variables from the Poissonization process with $N_\Pois \sim \Pois(n)$ for some parameter $n$; see \cref{sec:poissonization}.
Since $N_\Pois = \sum_{\ba} N_\ba = \sum_{\ba, \bx} N_{\bx,\ba} = \sum_{\ba, \bx, \by} N_{\by,\bx,\ba}$, we see that $0 \leq \frac{N_\ba}{N_\Pois} \leq 1$ and $0 \leq \frac{N_{\by,\bx,\ba}}{N_{\bx,\ba}} \leq 1$ for each of these fractional terms.

We can define a threshold $\tau > 0$ and partition the values of $\bA$ accordingly, and then define three summations $J_{< \tau}$, $J_{\geq \tau}$, and $K$ so that $T_{\bA,\bx,\by} - \hatT_{\bA,\bx,\by} = J_{< \tau} + J_{\geq \tau} + K$:
\begin{align*}
\bSigma_{\bA \geq \tau} &= \left\{ \ba \in \bSigma_{\bA} : \bbP(\bx, \ba) \geq \tau \right\}\\
J_{< \tau} &= \sum\nolimits_{\ba \not\in \bSigma_{\bA \geq \tau}} \bbP(\ba) \cdot \left( \bbP(\by \mid \bx, \ba) - \frac{N_{\by,\bx,\ba}}{N_{\bx,\ba}} \right)\\
J_{\geq \tau} &= \sum\nolimits_{\ba \in \bSigma_{\bA \geq \tau}} \bbP(\ba) \cdot \left( \bbP(\by \mid \bx, \ba) - \frac{N_{\by,\bx,\ba}}{N_{\bx,\ba}} \right)\\
K &= \sum\nolimits_\ba \left( \bbP(\ba) - \frac{N_\ba}{N_\Pois} \right) \cdot \frac{N_{\by,\bx,\ba}}{N_{\bx,\ba}}
\end{align*}
Since $\alpha_{\bA} = \min_{\ba \in \bSigma_{\bA}} \bbP(\bx \mid \ba)$, we see that $\bbP(\ba) \leq \frac{\tau}{\alpha_{\bA}}$ for $\ba \not\in \bSigma_{\bA \geq \tau}$.
One can show that $|J_{< \tau}| \leq \frac{\tau \cdot |\bSigma_{\bA}|}{\alpha_{\bA}}$ using triangle inequality and definition of $\bSigma_{\bA \geq \tau}$.
As for $|J_{\geq \tau}|$ and $|K|$, we condition on the concentration event $\cE^{J}_{\geq \tau} = \bigcap_{\ba \in \bSigma_{\bA \geq \tau}} \left\{ N_{\bx, \ba} > \frac{n \cdot \bbP(\bx, \ba)}{2} \right\}$ which holds with probability at least $1 - |\bSigma_{\bA}| \cdot \exp \left( -\frac{n \tau}{12} \right)$.
Under the event $\cE^{J}_{\geq \tau}$, one can show that $\Pr \left( \left| J_{\geq \tau} \right| > t \mid \cE^{J}_{\geq \tau} \right) \leq 2 \exp \left( - 2 n \alpha_{\bA} t^2 \right)$ and we can reduce the analysis of $|K|$ to the problem of producing an $\eps$-close estimate of $\bbP(\bA)$ by observing that $0 \leq \frac{N_{\by,\bx,\ba}}{N_{\bx,\ba}} \leq 1$ and $\frac{N_\ba}{N_\Pois}$ is the empirical estimate of $\bbP(\ba)$ for each $\ba \in \bSigma_{\bA}$.
That is, $|K| \leq \sum_\ba \left| \bbP(\ba) - \wh{\bbP}(\ba) \right|$.
The claim follows by putting together the above discussed bounds appropriately.
See \cref{sec:appendix-sample-complexity} for details.
\end{proofsketch}

\section{Finding small adjustment sets via approximate Markov blankets}
\label{sec:markov-blanket}

As discussed in \cref{sec:intro-main-results}, the bound in \cref{thm:estimation-error} motivates the use of small adjustment sets whenever possible.
In this section, we focus on searching for such sets.
In \cref{sec:misspecification-error-amb}, we describe our proofs of Lemmas \ref{lem:misspecification-error} and \ref{lem:misspecification-error-lower-bound} on the misspecification error of performing covariate adjustment with an approximate Markov blanket.
In \cref{sec:AMBA}, we describe our proof of \cref{thm:AMBA} on the sample complexity of approximate Markov blanket discovery; similarly, \cref{sec:BAMBA} describes our proof of \cref{thm:BAMBA}, which searches for even smaller adjustment sets.
\cref{sec:putting-together} concludes by putting these results together into a combination algorithm and describing the proof of our final result, \cref{thm:pac-causal-effect-estimation-special-case}.

\subsection{Misspecification error for approximate Markov blankets}\label{sec:misspecification-error-amb}

We begin with \cref{lem:misspecification-error}, which extends the equality $T_{\bS,\bx,\by} = T_{\bA,\bx,\by}$ for exact Markov blankets to a bound on the misspecification error $|T_{\bS,\bx,\by} - T_{\bA,\bx,\by}|$ for approximate Markov blankets.
To accompany this result, we also give a matching lower bound on the misspecification error in \cref{lem:misspecification-error-lower-bound}.

\begin{proofsketch}[\cref{lem:misspecification-error}]
By \cref{eq:T-alternative} and $\bS \subseteq \bA$, we see that
\[
\left| T_{\bS,\bx,\by} - T_{\bA,\bx,\by} \right| = \left| \sum_{\ba} \bbP(\by \mid \ba, \bx) \cdot \bbP(\ba \setminus \bs \mid \bs, \bx) \cdot \bbP(\bs) - \sum_{\ba} \bbP(\by \mid \ba, \bx) \cdot \bbP(\ba) \right|
\]
Then by triangle inequality, non-negativity of probabilities, and since $\bbP(\by \mid \ba, \bx) \leq 1$, \cref{eq:alpha-def}, and \cref{eq:approx-cond-ind-alternative}, one can show that this is at most $\frac{\eps}{\alpha_{\bS}}$.
See \cref{sec:appendix-misspecification-errors} for a step-by-step derivation.
\end{proofsketch}

\begin{proofsketch}[\cref{lem:misspecification-error-lower-bound}]
One can construct a distribution $\bbP$ defined over binary variables $\{X,Y,A,B\}$ with (conditional) probabilities below, which implies the four properties of the claim.
\begin{center}
\begin{tabular}{cc|cccc}
\toprule
$a$ & $b$ & $\bbP(b \mid a)$ & $\bbP(X = 0 \mid a,b)$ & $\bbP(X = 0 \mid a)$ & $\sum_x | \bbP(x \mid a,b) - \bbP(x \mid a)|$\\
\midrule
$0$ & $0$ & $\sqrt{\eps}/2$ & $1 - \alpha + \sqrt{\eps}/2$ & $1 - \alpha + \eps/4$ & $\sqrt{\eps} - \eps/2$\\
$0$ & $1$ & $1-\sqrt{\eps}/2$ & $1 - \alpha$ & $1 - \alpha + \eps/4$ & $\eps/2$\\
$1$ & $0$ & $1-\sqrt{\eps}/2$ & $\alpha$ & $\alpha - \eps/4$ & $\eps/2$\\
$1$ & $1$ & $\sqrt{\eps}/2$ & $\alpha - \sqrt{\eps}/2$ & $\alpha - \eps/4$ & $\sqrt{\eps} - \eps/2$\\
\bottomrule
\end{tabular}
\end{center}
where $\bS = \{A\} \subset \{A,B\} = \bZ$.
See \cref{sec:appendix-misspecification-errors} for detailed calculations and derivations.
\end{proofsketch}

\subsection{Approximate Markov blankets: discovery and adjustment}
\label{sec:AMBA}

We now prove \cref{thm:AMBA}, our upper bound on the sample complexity of finding an $\eps$-Markov blanket of $\bX$ with respect to an arbitrary set $\bA \subseteq \VminusXY$ using \cref{alg:AMBA}.

\begin{algorithm}[htb]
\caption{\AMBA\ (\textsc{AMBA})}
\label{alg:AMBA}
\KwData{$\eps, \delta > 0$, dataset $\bD$ of $n$ i.i.d. samples from $\bbP(\bV)$, and subset $\bA \subseteq \bV$}
\KwResult{$\bS \subseteq \bA$}
\For{$k = 0, 1, 2, \ldots, |\bA|$} {
    Let $w_k = \left( |\bA| \cdot \binom{|\bA|}{k} \right)^{-1}$\;
    
    Let $\bC_k = \big\{ \bS \subseteq \bA : |\bS| = k$, where \;
    
    \hspace{95pt}$\ApproxCondInd(\bX \ind \bA \setminus \bS \mid \bS, \eps, \delta w_k, \bD)$ outputs YES$\big\}$\;
    
    \If{$|\bC_k| > 0$} {
        \Return any $\bS \in \bC_k$\;
    }
}
\Return $\bA$\;
\end{algorithm}


\begin{myproof}[\cref{thm:AMBA}]
Suppose \AMBAshort\ (\cref{alg:AMBA}) terminates at iteration $|\bS| \in \{0, 1, \ldots, |\bA|\}$.

\paragraph{Correctness.}
Suppose all calls to \ApproxCondInd\ succeed, then \cref{lem:approxcondind-guarantees} tells us that any produced $\bS \subseteq \bA$ satisfies the property that $\Delta_{\bX \ind \bA \setminus \bS \mid \bS} \leq \eps$.
\cref{lem:misspecification-error} then further tells us that $\Delta_{\bX \ind \bA \setminus \bS \mid \bS} \leq \eps$ implies $| T_{\bS, \bx, \by} - T_{\bA, \bx, \by} | \leq \frac{\eps}{\alpha_{\bS}}$.

\paragraph{Failure rate.}
There are at most $\binom{|\bA|}{k}$ possible candidate sets in $\bC_k$ for each $k \in \{0, 1, \ldots, |\bA|\}$.
Since we invoked each call to $\ApproxCondInd$ with $\delta w_k$ in iteration $k$, union bound tells us that the probability of \emph{any} call failing across all calls is at most
\[
\sum_{k=0}^{|\bS|} \delta w_k \cdot \binom{|\bA|}{k}
= \sum_{k=0}^{|\bS|} \delta \cdot \frac{1}{|\bA| \cdot \binom{|\bA|}{k}} \cdot \binom{|\bA|}{k}
= \sum_{k=0}^{|\bS|} \frac{\delta}{|\bA|}
\leq \frac{\delta \cdot |\bS|}{|\bA|}
\leq \delta
\]

\paragraph{Sample complexity.}
Since we use a union bound to bound our overall failure probability, we can reuse samples in all our calls to \ApproxCondInd.
Thus, the total sample complexity is dominated by the final call (where $k = |\bS|$), which uses $\wt{\cO} \left( \frac{1}{\eps^2} \cdot \sqrt{|\bSigma_{\bX}| \cdot |\bSigma_{\bA \setminus \bS}| \cdot |\bSigma_{\bS}|} \cdot \log \frac{1}{\delta w_k} \right)$ samples according to \cref{lem:approxcondind-guarantees}. 
Plugging in $w_k = \left( |\bA| \cdot \binom{|\bA|}{k} \right)^{-1}$, we obtain total sample complexity $\wt{\cO} \left( \frac{|\bS|}{\eps^2} \cdot \sqrt{|\bSigma_{\bX}| \cdot |\bSigma_{\bA}|} \cdot \log \frac{1}{\delta} \right)$, where we omit $\log |\bA|$ within $\wt{\cO}(\cdot)$ because $|\bA| \leq |\bSigma_{\bA}|$.
\end{myproof}

\subsection{Beyond approximate Markov blankets}
\label{sec:BAMBA}

Motivated by \cref{fig:intro-motivation}, which shows that the Markov blanket of $\bX$ with respect to $\bZ$ may still be large compared to the smallest adjustment set, we study in this section an approach for finding smaller adjustment sets than the Markov blanket.
We prove \cref{lem:minimal-soundness} in \cref{sec:appendix-BAMBA}, which establishes conditions on sets $\bS' \subseteq \VminusXY$ and $\bS \subseteq \VminusXY$ such that $T_{\bS',\bx,\by} = T_{\bS,\bx,\by}$; this result suggest an approach for going beyond adjustment by Markov blankets.
This allows us to show \cref{thm:BAMBA}, our upper bound on the sample complexity of finding a set $\bS'$ that approximately satisfies the conditions of \cref{lem:minimal-soundness} with respect to an $\eps$-Markov blanket $\bS$.

\begin{algorithm}[htb]
\caption{\BAMBA\ (\BAMBAshort)}
\label{alg:BAMBA}
\KwData{$\eps, \delta > 0$, $n$ i.i.d. samples $\bD$ from $\bbP(\bV)$, subset $\bA \subseteq \bV$, and $\eps$-Markov blanket $\bS \subseteq \bA$}
\KwResult{$\bS' \subseteq \bA$ such that $|\bSigma_{\bS'}| \leq |\bSigma_{\bS}|$}
\For{$k = 0, 1, 2, \ldots, |\bA|$} {
    Let $w_k = \left( |\bA| \cdot \binom{|\bA|}{k} \right)^{-1}$\;
    
    Let $\bC_k = \big\{ \bS' \subseteq \bA : |\bS'| = k$ and $|\bSigma_{\bS'}| \leq |\bSigma_{\bS}|$, where\;
    
    \hspace{97pt}$\ApproxCondInd(\bY \ind \bS \setminus \bS' \mid \bX \cup \bS', \eps, \frac{\delta w_k}{2}, \bD)$ outputs YES,\;
    
    \hspace{97pt}$\ApproxCondInd(\bX \ind \bS' \setminus \bS \mid \bS, \eps, \frac{\delta w_k}{2}, \bD)$ outputs YES$\big\}$\;
    
    \If{$|\bC_k| > 0$} {
        \Return any $\bS' \in \bC_k$\;
    }
}
\Return $\bS$\;
\end{algorithm}


\begin{proofsketch}[\cref{thm:BAMBA}]
The proof follows the same structure as that of \cref{thm:AMBA}.
The key difference is that in the correctness analysis, we apply triangle inequality $\left| T_{\bS, \bx, \by} - T_{\bS', \bx, \by} \right| \leq \left| T_{\bS, \bx, \by} - Z_{\bx, \by} \right| + \left| Z_{\bx, \by} - T_{\bS', \bx, \by} \right|$ for some intermediate term $Z_{\bx, \by} = \sum_{\bs \cup \bs'} \frac{1}{\bbP(\bx \mid \bs)} \cdot \bbP(\bx, \bs \cup \bs') \cdot \bbP(\by \mid \bx, \bs')$ which allows us to relate the two approximate conditional independences to $\alpha_{\bS}$.
The claim follows after upper bounding each of the terms by $\frac{\eps}{\alpha_{\bS}}$.
See \cref{sec:appendix-BAMBA} for details.
\end{proofsketch}

\subsection{A combination algorithm}
\label{sec:putting-together}

In \cref{sec:appendix-putting-together}, we re-express the results of \cref{thm:estimation-error}, \cref{thm:AMBA} and \cref{thm:BAMBA} in terms of an upper bound on error for a fixed number of samples $n$.
After which, there are a couple of ways one could attempt to estimate $\bbP_{\bx}(\by)$ when given a valid adjustment set $\bZ \subseteq \VminusXY$:\\
\phantom{\hspace{10pt}}1. Directly estimate using $\bZ$.\\
\phantom{\hspace{10pt}}2. Use \AMBAshort\ on $\bZ$ to produce a subset $\bS \subseteq \bZ$ and estimate using $\bS$.\\
\phantom{\hspace{10pt}}3. Use \AMBAshort\ on $\bZ$ to produce a subset $\bS \subseteq \bZ$, then use \BAMBAshort\ to further produce subset $\bS'$,\\
\phantom{\hspace{22pt}}and then estimate using $\bS'$.\\
As we show in \cref{sec:appendix-putting-together}, the second approach yields an asymptotically smaller error when $|\bS| \cdot \sqrt{\frac{|\bSigma_{\bX}|}{|\bSigma_{\bZ}|}} < \max \left\{\frac{|\bSigma_{\bZ}|}{n}, \frac{\alpha_{\bS}}{|\bSigma_{\bZ}|}, \alpha_{\bS}^2 \right\}$; a similar kind of decision happens for deciding whether to use the third approach.
The proof of \cref{thm:pac-causal-effect-estimation-special-case} follows by a careful combination of these insights.

\begin{proofsketch}[\cref{thm:pac-causal-effect-estimation-special-case}]
Consider the following algorithm:\\
\phantom{\hspace{10pt}}1. Run \AMBAshort\ to obtain $\bS \subseteq \bZ$\\
\phantom{\hspace{10pt}}2. Check if $|\bS| \cdot \sqrt{\frac{|\bSigma_{\bX}|}{|\bSigma_{\bZ}|}} < \max \left\{ \frac{|\bSigma_{\bZ}|}{n}, \frac{\alpha_{\bS}}{|\bSigma_{\bZ}|}, \alpha_{\bS}^2 \right\}$\\
\phantom{\hspace{10pt}}3. If so, run \BAMBAshort\ to obtain $\bS' \subseteq \bZ$ and produce estimate $\hatbbP_{\bx}(\by) = \hatT_{\bS',\bx,\by}$\\
\phantom{\hspace{10pt}}4. Otherwise, produce estimate $\hatbbP_{\bx}(\by) = \hatT_{\bZ,\bx,\by}$\\
That is, depending on the check, we decide to perform estimation based on $\bS^* = \bS'$ or $\bS^* = \bZ$.
One can show that the bound holds for each case separately while noting that $\alpha_{\bS}, \alpha_{\bS'}, \alpha_{\bZ} \geq \alpha$.
\end{proofsketch}

\section{Conclusion}\label{sec:conclusion}

We now provide a brief final summary of our contributions.
In this paper, we have focused on the problem of estimating the causal effect $\bbP_\bx(\by)$ in the PAC setting, given access to a valid adjustment set $\bZ$, i.e., $\bZ$ such that $\bbP_\bx(\by) = T_{\bZ,\bx,\by}$, defined in \cref{eqn:covariate-adjustment}.

\begin{tabular}{@{}cp{0.9\linewidth}@{}}
(1) & In \cref{sec:sample-complexity}, we established a PAC bound for estimation of $T_{\bA,\bx,\by}$ for an arbitrary set $\bA \subseteq \VminusXY$, with linear dependence on $|\bSigma_\bA|$, the alphabet size of $\bA$.\\
(2) & In \cref{sec:AMBA}, we established a bound on the misspecification error $|T_{\bS,\bx,\by} - T_{\bA,\bx,\by}|$ for $\bS$ which is an $\eps$-Markov blanket of $\bX$ with respect to $\bZ$, and a PAC bound for discovering such a set $\bS$; leading to a new estimator of $T_{\bA,\bx,\by}$ with reduced sample complexity.\\
(3) & In \cref{sec:BAMBA}, we established conditions under which $T_{\bS',\bx,\by} = T_{\bS,\bx,\by}$, and gave a PAC bound for discovering a set $\bS'$ under which these conditions approximately hold; leading to a new estimator of $T_{\bA,\bx,\by}$ which goes beyond using the Markov blanket for adjustment.
\end{tabular}




Furthermore, in \cref{sec:related-work}, we review related work, give further interpretations of our results under the graphical causality framework, and connect our results to existing results in this line of work.
These results pave the way for future connections between causal discovery and causal effect estimation, while each standing alone as results of independent interest for fields such as local causal discovery.
We conclude with a non-exhaustive list of open problems raised by our work, which we expect to be of immediate future interest.

\begin{tabular}{@{}cp{0.9\linewidth}@{}}
(1) & Compared to the expectation bound of \cite{zeng2024causal}, our PAC bounds contains an additional $\wt{\cO}\left( \frac{|\bSigma_\bA|}{\eps^2} \right)$ term. Can this term be eliminated or shown to be necessary?\\
(2) & Our \AMBAshort\ algorithm for $\eps$-Markov blanket discovery performs an exhaustive search over subsets of increasing size.
Fortunately, this search is ``embarrassingly parallel", but is computationally prohibitive without access to parallel computing.
Is there a more computationally efficient algorithm for this problem with (nearly) the same sample complexity?
\end{tabular}




\acks{
This research/project is supported by the National Research Foundation, Singapore under its AI Singapore Programme (AISG Award No: AISG-PhD/2021-08-013).
This study was supported in part by Office of Naval Research Award No. N00014-21-1-2807.
The authors would like to thank Karthikeyan Shanmugam, Murat Kocaoglu, Raghavendra Addanki, Emilija Perkovi{\'c}, and Bijan Mazaheri for useful discussions, correspondence, and feedback on earlier draft versions.
}

\newpage
\bibliography{bib}

\newpage
\clearpage
{
\hypersetup{linkcolor=blue}
\renewcommand{\contentsname}{Contents of Appendix}
\tableofcontents
}
\addtocontents{toc}
{\protect\setcounter{tocdepth}{2}}

\newpage
\appendix

\section{Related work}
\label{sec:related-work}

For sake of clarity, we have positioned this work from the perspective of causal effect estimation, and emphasized how our primary assumption (knowledge of a valid adjustment set $\bZ$) is compatible with both the potential outcomes (PO) and graphical frameworks for causality.
Now, we further explore connections between our work and existing work from both of these perspectives, beginning with the graphical perspective.

In \cref{sec:related-work-causal-effect-estimation}, we review graphical characterizations of valid adjustment sets given a causal graph (or an equivalence class of graphs) as input.
In some domains, these causal graphs may be constructed from expert knowledge, but when $\bV$ is large or the system under consideration is not well-studied, practitioners may be unable to specify an accurate causal graph.
Thus, we also review conditions for causal effect estimation which require minimal graphical knowledge.
In \cref{sec:related-work-causal-graph-discovery}, we review a different approach to the unspecified graph setting; in particular, we discuss methods for learning all or part of a causal graph from data, and relate these to our work by providing graphical interpretations of our \AMBAshort\ and \BAMBAshort\ methods.
Finally, we pivot to the PO perspective in \cref{sec:related-work-feature-selection}, focusing on existing results on the statistical aspects of causal effect estimation.

\subsection{Causal effect identification in the graphical setting}
\label{sec:related-work-causal-effect-estimation}

In the graphical framework, several classes of graphs have been used to formally define causal assumptions about a system, with the nodes of these graphs corresponding to the observed variables $\bV$.
Here, we focus our discussion on acyclic directed mixed graphs (ADMGs), which consist of both directed edges (of the form $V_1 \to V_2$) and bidirected edges (of the form $V_1 \leftrightarrow V_2$).
Such graphs can be used to model systems that are subject latent (a.k.a.\ unobserved) confounding, but which are not subject to selection bias.
As a special case, an ADMG with no bidirected edges is called a directed acyclic graph (DAG), representing a system with no latent confounding.
To distinguish between these cases, we use the term \textit{causally sufficient} to refer to settings where latent confounding is assumed to be absent; thus \textit{causally insufficient} refers to settings where latent confounding is allowed to exist.
Finally, given an ADMG $\cG$, an intervention that sets the variables $\bX \subseteq \bV$ equal to the values $\bx$ can be represented by a new ADMG, the \textit{mutilated graph} $\cG_{\overline{\bX}}$, which is obtained by copying $\cG$ and then removing all edges of the form $V \to X$ or $V \leftrightarrow X$ for an $X \in \bX$ and $V \in \bV$.

\subsubsection{Causal effect identification given a graph}

The graph $\cG$ and $\cG_{\overline{\bX}}$ can be used to model the behavior of the system, and to derive relationships between the observational distribution $\bbP(\bV)$ and the interventional distribution $\bbP_\bx(\bV)$.
The details of these definitions are not necessary for our discussion; instead, we describe some of the major results which have been shown when taking these definitions as a starting point.
Most importantly for our discussion, these definitions can be used to derive \textit{identification formulas}, which express interventional queries $\bbP_\bx(\by)$ in terms of equations which only involve $\bbP(\bV)$, and thus permit causal effects to be estimated from only observational data.
These identification formulas can be derived algorithmically, for example using the \textit{ID Algorithm} \citep{tian2002general}, which is both sound and complete \citep{shpitser2006identification,huang2006pearl}.
PAC bounds have also been established for the ID algorithm in \citep{bhattacharyya2022efficient}.

Importantly, the ID Algorithm may be able to construct an identification formula even if the adjustment formula (\cref{eqn:covariate-adjustment}) does not hold for any set $\bA \subseteq \VminusXY$.
However, in practice, the adjustment formula remains one of the most widely-used and well-studied identification approaches, due in part to its simplicity and its familiarity in the potential outcomes literature (see \cref{sec:appendix-potential-outcomes-adjustment}).
Particular attention has been given to developing graphical criteria for determining whether a set $\bA \subseteq \VminusXY$ is a valid adjustment set for $\bbP_\bx(\by)$, and algorithmically finding such a set if one exists.
A simple and intuitive condition for adjustment validity is the \textit{backdoor criterion} \citep{pearl1995causal} in DAGs, which is sound, but not complete.
This criterion has been refined by long line of work on sound and complete conditions \citep{shpitser2010validity,van2014constructing,maathuis2015generalized,perkovic2018complete,perkovic2020identifying} for different classes (and equivalence classes) of causal graphs.
Our results further contribute to this line of work: as we discuss in \cref{sec:introduction}, \cref{lem:minimal-soundness} directly implies a graphical condition that is sound for determining whether a subset is an adjustment set given a valid adjustment set; \cref{sec:appendix-completness} shows that under additional assumptions, this condition is also complete.

\subsubsection{Causal effect identification without a graph}

While the criteria above are stated in terms of a known causal graph $\cG$, they can also be used in our setting to derive conditions under which \cref{eqn:covariate-adjustment} holds, even when the graph is an unknown.
Indeed, using \cref{eqn:covariate-adjustment} requires quite minimal background knowledge of $\cG$, as we now discuss.
For simplicity, we limit our discussion to a single treatment variable $X$. 
In the case of DAGs, the backdoor criterion implies that $\bZ = \ND(X)$ is a valid adjustment set, where $\ND(X)$ denotes the set of non-descendants of $X$ in $\cG$.
Thus, assuming causal sufficiency, our method can be employed given only knowledge of $\ND(X)$, a quite common setting in applications such as healthcare, where a doctor's treatment assignment $X$ can only depend on pre-treatment patient covariates.
Under causal sufficiency and $\bZ = \ND(X)$, the Markov blanket of $X$ with respect to $\bZ$ is the set $\bS = \Pa(X)$, and our \AMBAshort\ algorithm can be interpreted as searching for the parents of $X$.
Similar results hold in the more general case of ADMGs under an additional assumption on the graph $\cG$, as we discuss in \cref{sec:appendix-ADMG}.

In light of these connections, our results fit into a recent line of work establishing identifiability of causal effects with minimal graphical background knowledge.
\citet{entner2013data} consider a setting that matches ours in the DAG setting with $\bZ = \ND(X)$, and establish a condition similar to \cref{lem:minimal-soundness} to determine whether $\bA \subseteq \VminusXY$ is a valid adjustment set.
While our condition is sound, their condition is both sound \textit{and} complete, but relies on conditional dependence checks instead of only conditional independence checks.
Furthermore, in contrast with our work, where statistical guarantees are a primary focus, their work does not provide any guarantees outside of the oracle setting, though it would be interesting to study their approach in the finite-sample setting.

Follow-up works in this space have extended this problem to the causally insufficient setting by incorporating additional background knowledge on $\cG$; all of the works discussed assume knowledge of $\bZ = \ND(X)$.
For example, \citet{cheng2022toward} assumes knowledge of some variable $A$ that is a ``cause or spouse of treatment only (COSO)" variable, i.e.\ that $A$ is adjacent to $X$ but not to $Y$ in $\cG$, and establishes a sound condition for determining whether $\bS \subseteq \bZ$ is an adjustment set.
Relatedly, \citet{shah2022finding} assumes knowledge of some variable $A$ that is a parent of $X$ and establishes a similar condition.
Both conditions are sound, but not complete; in contrast, we show in \cref{sec:appendix-completness} that the \BAMBAshort\ approach is both sound and complete in the causally sufficient setting when $\bZ = \ND(X)$.
Finally, \citet{shah2023front} goes beyond using the adjustment formula for identification, in particular studying when background knowledge is sufficient to identify the causal effect using \textit{frontdoor} adjustment.

\subsection{Causal graph discovery}
\label{sec:related-work-causal-graph-discovery}

This work is strongly motivated by our recognition of the pressing need for better connections between the areas of causal effect estimation and causal structure learning.
In a typical \textit{causal discovery} (a.k.a.\ \textit{causal structure learning}) task, one takes data on the observed variables $\bV$ as input, and seeks to return a causal graph $\cG$ (or an equivalence class of graphs) that provides an accurate causal model of the system.
Traditionally, this goal is (implicitly or explicitly) motivated by the utility of such a model for generating causal predictions, e.g., predicting $\bbP_\bx(\by)$ as discussed in this work.

\subsubsection{Causal discovery and faithfulness}
The field of causal discovery is quite well-developed, and has been the subject of several surveys, e.g. see \citep{heinze2018causal,glymour2019review,vowels2022d,squires2023review}.
Various approaches address settings such as learning from observational data in the causally sufficient setting \citep{spirtes2000causation,chickering2002optimal,zheng2018dags,solus2021consistency} and in the causally insufficient setting \citep{spirtes2000causation,colombo2012learning}, as well as learning from interventional data, possibly involving actively chosen interventions \citep{eberhardt2005number, eberhardt2006n, eberhardt2007causation, hauser2012characterization, hu2014randomized, shanmugam2015learning, wang2017permutation, kocaoglu2017cost, lindgren2018experimental, greenewald2019sample, jaber2020causal, squires2020active, choo2022verification, choo2023subset, choo2023active, choo2023new, choo2023adaptivity}.
Many of these algorithms enjoy theoretical guarantees in the well-specified setting, i.e.\ under the assumption that the system is correctly described by some (unknown) causal graph $\cG^*$.
In this setting, an algorithm is called \textit{consistent} if it recovers $\cG^*$ (or an appropriate equivalence class) with probability one in the limit of infinite data.

Significant attention has been devoted to finding conditions under which various causal discovery algorithms are consistent.
For example, the well-known \textit{faithfulness} assumption requires that if $\bA$ and $\bB$ and not d-separated by $\bC$ in $\cG^*$, then $\bA \not\ind \bB \mid \bC$ in $\bbP(\bV)$.
Although faithfulness is a sufficient condition for the consistency of many causal discovery algorithms, it is often a stronger condition than necessary, and many weaker conditions have been established, see \citep{lam2023causal} for a recent review and comparison of such conditions.
The search for weaker consistency conditions is motivated by a practical issue: although the consistency of an algorithm may depend only on there being no violations of faithfulness, \textit{near} violations of faithfulness (where the conditional independence $\bA \ind \bB \mid \bC$ nearly holds, e.g. $\Delta_{\bA \ind \bB \mid \bC} \leq \eps$ for some small $\eps$) can significantly affect its finite sample properties.
Therefore, finite sample guarantees for graph recovery \citep{kalisch2007estimating,maathuis2009estimating,gao2020polynomial,wadhwa2021sample,gao2022optimal} often depend on assumptions such as \textit{strong} faithfulness, which may be significantly more restrictive in practice \citep{uhler2013geometry}.

In this work, we avoid making any such assumptions.
Indeed, since our goal is causal effect estimation, rather than graph recovery, faithfulness conditions are unnecessary, and existing sample complexity guarantees for causal discovery are pessimistic for our purposes.
Within the graphical framework, a main message of our work is that accurate causal effect estimation does not require learning the correct causal graph $\cG^*$.
For example, if $\cG^*$ has ``weak" edges, these may be hard to distinguish from missing edges, but those edges are also exactly those that do not significantly impact causal effects; in pragmatic terms, whether an edge is weak or missing is ``a difference that doesn't make a difference".
We provide a concrete example of this phenomenon in \cref{sec:appendix-weak-edges}.
Nonetheless, such conditions may be useful in improving the sample complexity and/or the computational complexity of our approach, as we discuss in \cref{sec:appendix-faithfulness}.

\subsubsection{Cautious approaches and local causal discovery}
To better align theory and practice, a few recent works have focused on new kinds of theoretical guarantees.
Two contemporaneous works \citep{malinsky2024cautious,chang2024post} explicitly consider the interplay between causal discovery and causal effect estimation.
As in our work, \citet{malinsky2024cautious} advocates the use of conditional dependence tests (as opposed to conditional independence tests) to control model misspecification, an approach they call ``cautious" causal discovery, where \citet{chang2024post} advocate a bootstrap-style approach.
However, their guarantees are for the asymptotic setting, rather than the PAC setting considered in this work, and their approaches aim to recover an entire causal graph, unlike our approach.

More closely related to our approach are methods for \textit{local causal discovery}, which aim to recover only part of a causal graph.
Indeed, one of the canonical problems in local discovery is Markov blanket recovery \citep{koller1996toward,frey2003identifying,tsamardinos2003algorithms,ramsey2006pc,pena2007towards,fu2008fast,aliferis2010local,aliferis2010local2,
gao2016efficient,ling2020using,dong2022nonparametric}, potentially combined with partial edge orientation \citep{yin2008partial,wang2014discovering,gao2015local,gupta2023local} and often used in the context of full causal discovery algorithms \citep{mani2004causal,tsamardinos2006max,solus2021consistency,gao2021efficient}.
A number of these algorithms employ greedy search, adding variables to the Markov blanket one at a time, e.g.\ \citep{tsamardinos2003algorithms,fu2008fast,gao2016efficient}.
However, greedy search is not guaranteed to return a correct Markov blanket without additional assumptions, such as those in \citep{gao2021efficient}, in which the authors also provide finite sample guarantees.
In contrast, many non-greedy algorithms do enjoy consistency guarantees (i.e. recovery of a correct Markov blanket in the infinite data limit), but thus far lack finite sample guarantees.

Thus, our finite sample guarantees for the (non-greedy) \AMBAshort\ algorithm contribute to this important line of work, and may be of independent interest beyond the context of causal effect estimation.
Furthermore, our \BAMBAshort\ highlights that using only local structure may be suboptimal for some estimation problems.
This fact suggests that we extend from local causal discovery to the more general problem of \textit{targeted} causal discovery, i.e., causal discovery tailored to specific estimation problems, analogous to techniques such as targeted maximum likelihood estimation \citep{van2006targeted,schuler2017targeted}.

\subsection{Causal effect estimation via covariate adjustment}
\label{sec:related-work-feature-selection}

Now, we relate our results to existing statistical results on causal effect estimation, focusing on estimation using the adjustment formula.
Existing results are largely written in terms of potential outcomes but, as with our result, are usually applicable as long as \cref{eqn:covariate-adjustment} holds and are thus independent of framework choice.\footnotemark
\footnotetext{
    When the random variables are continuous or mixed, \cref{eqn:covariate-adjustment} is written as $T_{\bs,\bx,\by} = \bbE_\bS [ \bbP(\bY = \by \mid \bX = \bx, \bS) ]$.
}
In many domains such as healthcare and econometrics, \cref{eqn:covariate-adjustment} can be justified by domain knowledge.
For example, in healthcare, where $\bX$ and $\bY$ may represent medical treatments and patient outcomes, respectively, it is sufficient for $\bZ$ to contain all information that doctors may be using to assign treatment, e.g.\ patient demographic information and past medical history.
In such domains, $\bZ$ are often referred to as a set of \textit{covariates}; we adopt this terminology here.

As datasets become larger and richer, causal effect estimation is increasingly being applied to problems with high-dimensional covariates.
These problems present novel challenges, including violations of the overlap assumption \citep{d2021overlap} and the breakdown of traditional asymptotic results.
Dimensionality reduction techniques such as feature selection are often crucial to addressing the challenges.
However, in the context of treatment effect estimation, na{\"i}ve usage of feature selection methods such as the Lasso can introduce substantial misspecification bias.
Several works aim to address this issue; here, we focus on methods based on feature selection, pointing readers to \citet{yadlowsky2022bounds} as a starting point for methods using other forms of dimensionality reduction, and to \citet{witte2019covariate} and \citet{yu2020causality} for a more complete review and comparison of methods based on feature selection.

Whereas our work focuses on discrete covariates, with no additional assumptions on $\bbP(\bZ)$, $\bbP(\bX \mid \bZ)$ and $\bbP(\bY \mid \bX, \bZ)$, the majority of prior works consider \textit{continuous} covariates $\bZ$, and thus require additional assumptions, such as parametric or smoothness assumptions.
When $X$ is a binary treatment, a common assumption is that $\bbP(X \mid \bZ)$ follows a logit model, so that $\bbP(X \mid \bZ)$ is parameterized by a vector $\bbeta \in \bbR^{|\bZ|}$.
Similarly, when $Y$ is a scalar outcome, a common assumption is that $\bbP(Y \mid X, \bZ)$ follows a linear model, i.e.\ it is parameterized by a vector $\bgamma \in \bbR^{|\bZ|}$.
Sparsity assumptions may be imposed on one or both of $\bbeta$ and $\bgamma$; for example, \citet{shortreed2017outcome} and \citet{wang2020debiased} assume sparsity on $\bbeta$, \citet{bradic2019sparsity} and \citet{athey2018approximate} assume sparsity on $\bgamma$, and \citet{greenewald2021high} assumes sparsity on both.
Other common assumptions include semiparametric restrictions, e.g.\ partially linear models \citep{belloni2014inference,chernozhukov2018double}, and smoothness assumptions \citep{farrell2021deep}.

In these works, sparse regression methods (e.g.\ Lasso and its variants) play a role similar to our search for a smaller adjustment set $\bS \subseteq \bZ$, and the choice of regularization parameter plays a role similar to our choice of $\eps$ in balancing between misspecification bias and estimation error.
In comparison to these methods, our focus on discrete variables obviates the need for additional assumptions, and allows us to establish deeper connections between causal effect estimation and fields such as distribution testing \citep{canonne2020survey} and property estimation \citep{charikar2019general}.
These connections make the problem accessible to a wider audience and provide access to a broader range of tools: in particular, we note that most of these prior results (e.g. \citet{shortreed2017outcome}, \citet{athey2018approximate}, \citet{bradic2019sparsity}, \citet{wang2020debiased}, \citet{belloni2014inference}, \citet{chernozhukov2018double}, and \citet{farrell2021deep}) are of an asymptotic nature, with \citet{greenewald2021high} being a key exception.

\section{Additional results}

\subsection{Derivation of expectation bound}
\label{sec:appendix-ZBHK24-derivation}

Here, we translate the result of \cite{zeng2024causal} into our language, showing that $\cO \left( \frac{|\bSigma_\bZ|}{\lambda \alpha_\bZ} + \frac{1}{\lambda^2 \alpha_\bZ} \right)$ samples suffice to obtain an expectation bound of $\bbE \left( \left| T_{\bZ,\bx,\by} - \hatT_{\bZ,\bx,\by} \right| \right) \leq \lambda$, for $T_{\bZ,\bx,\by}$ defined as in \cref{eqn:covariate-adjustment}.

\cite{zeng2024causal} studies the setting where one is given $n$ i.i.d.\ copies of $(Y, X, A)$ where $Y \in \{0,1\}$ is the binary outcome, $X \in \{0,1\}$ is the binary treatment, and $A \in [d] = \{1, \ldots, d\}$ is a multivariate covariate.
Under their positivity assumption \cite[Assumption 2]{zeng2024causal}, $\bbP(X = 1 \mid A = k) \in [\eps, 1 - \eps]$ holds for some constant $\eps \in (0, 1/2)$ and any $k \in [d]$.
Then, for $\psi_1 = \sum_{k=1}^d \bbP(A = k) \cdot \bbP(Y = 1 \mid X = 1, A = k)$ and plug-in estimator $\wh{\psi_1}$, Theorem 1 of \cite{zeng2024causal} states that $\bbE [\psi_1 - \wh{\psi_1} ] \leq \frac{|\bSigma_\bZ|^2}{\alpha_\bZ^2 n^2} + \frac{C}{\alpha_\bZ n}$ when $\wh{\psi_1}$ is computed using $n$ i.i.d. samples from $\bbP(Y, X, A)$, for the worst case distribution $\bbP(Y, X, A)$ satisfying their positivity assumption.

To adapt their result to our setting, let us define $Y' = \kron_{\bY = \by}$, $X' = \kron_{\bX = \bx}$, and $A'$ as a flattened version of $\bZ$.
Relating $(Y', X', A')$ to their $(Y, X, A)$ setup, we see that $\psi_1 = T_{\bZ,\bx,\by}$, $d = |\bSigma_{\bZ}|$, and $\alpha_{\bZ} = \eps$.
So,
\begin{equation}
\label{eq:ZBHK24-theorem1}
\bbE \left[ \left( T_{\bZ,\bx,\by} - \hatT_{\bZ,\bx,\by} \right)^2 \right]
\leq \frac{|\bSigma_\bZ|^2}{\alpha_\bZ^2 n^2} + \frac{C}{\alpha_\bZ n},    
\end{equation}
for some absolute constant $C > 0$, where we have replaced $\eps$ by $\alpha_\bZ$, $d$ by $|\bSigma_\bZ|$, and used that $(1 - \alpha_\bZ)^2 \leq 1$.

To translate this bound into our desired form, we first apply Jensen's inequality \cite{jensen1906fonctions}:
\begin{align*}
\left( \bbE \left[ \left| T_{\bZ,\bx,\by} - \hatT_{\bZ,\bx,\by} \right| \right] \right)^2
&\leq \bbE \left[ \left( \left| T_{\bZ,\bx,\by} - \hatT_{\bZ,\bx,\by} \right| \right)^2 \right] \tag{Jensen's inequality}\\
&= \bbE \left[ \left( T_{\bZ,\bx,\by} - \hatT_{\bZ,\bx,\by} \right)^2 \right]\\
&\leq 2 \max \left( \frac{|\bSigma_\bZ|^2}{\alpha_\bZ^2 n^2}, \frac{C}{\alpha_\bZ n} \right) \tag{By \cref{eq:ZBHK24-theorem1}}
\end{align*}
Thus, to obtain that $\bbE \left( \left| T_{\bZ,\bx,\by} - \hatT_{\bZ,\bx,\by} \right| \right) \leq \lambda$, it suffices to have
$
2 \max \left( \frac{|\bSigma_\bZ|^2}{\alpha_\bZ^2 n^2}, \frac{C}{\alpha_\bZ n} \right)
\leq
\lambda^2
$
.
Then, solving for $n$ yields
$
n \in \cO \left( \frac{|\bSigma_\bZ|}{\lambda \alpha_\bZ} + \frac{1}{\lambda^2 \alpha_\bZ} \right)
$
as stated.

\subsection{Derivation for conditional sub-Gaussian}
\label{sec:appendix-subgaussian-take-max}

For completeness, we present the following proof of \cref{lem:subgaussian-take-max}.

\begin{proof}
    By iterated expectation,
    \begin{align*}
        \bbE \left( e^{\lambda X} \right)
        &=
        \bbE \left( \bbE \left( e^{\lambda X} \mid Y \right) \right)
        \\
        &\leq
        \bbE \left( \exp \left( \frac{\lambda^2 \sigma_Y^2}{2} \right) \right)
        \\
        &\leq
        \bbE \left( \exp \left( \frac{\lambda^2 \max_{y \in \Sigma_Y} \sigma_y^2}{2} \right) \right)
        \\
        &\leq
        \exp \left( \frac{\lambda^2 \max_{y \in \Sigma_Y} \sigma_y^2}{2} \right),
    \end{align*}
    i.e., $X \in \subG(\max_{y \in \Sigma_Y} \sigma_y^2)$, as desired.
\end{proof}


\subsection{Additional results in the graphical framework}

In \cref{sec:related-work}, we described a special case of our setting in the graphical framework. In particular, assuming that $\cG$ is a DAG and considering only a single treatment variable $X$, it is easy to see that $\bZ = \ND(X)$ is a valid adjustment set, and that $\bS = \Pa(X)$ is a Markov blanket of $X$ with respect to $\bZ$.
We now discuss two other applications of our results.
First, in \cref{sec:appendix-completness}, we show that \BAMBAshort\ is complete in this special case, i.e.\ the two conditional independences in \cref{lem:minimal-soundness} are not just sufficient to ensure that $\bS'$ is an adjustment set, they are also \textit{necessary}.
Second, in \cref{sec:appendix-ADMG} we introduce an assumption under which we can extend the special case of $\bZ = \ND(X)$ from DAGs to ADMGs, and describe the Markov blanket blanket of $X$ with respect to $\bZ$.

We assume that the interested reader is familiar with definitions related to d-separation (in DAGs) and m-separation (in ADMGs); e.g.\ concepts such as colliders, active paths, and blocked paths; see e.g. \cite{richardson2003markov} for a detailed overview.
In particular, we will make use of the following Markov property for DAGs:
\begin{definition}\label{def:markov}
    A probability distribution $\bbP(\bV)$ is said to be \emph{Markov} with respect to a DAG $\cG$ if, whenever $\bA$ and $\bB$ are d-separated by $\bC$ in $\cG$, then $\bA$ is conditionally independent from $\bB$ given $\bC$ in $\bbP(\bV)$.
\end{definition}

\subsection{Completeness of \BAMBAshort\ for a special case}
\label{sec:appendix-completness}

In this section, we show that \BAMBAshort\ is not just sound (\cref{lem:minimal-soundness}) but also complete in a special setting.
In particular, \cref{lem:minimal-completeness} implies that searching for a minimal sized $\bS' \subseteq \bZ$ that satisfies both $\bY \ind \bS \setminus \bS' \mid \bX \cup \bS'$ and $\bX \ind \bS' \setminus \bS \mid \bS$ necessarily produces a minimal sized adjustment set.

\begin{mylemma}
\label{lem:minimal-completeness}
Consider the graphical causal framework in the causally sufficient setting, where variables in $\bX$ are non-ancestors of each other.
Let $\bZ = \ND(\bX) \subseteq \VminusXY$ be the set of non-descendants of $\bX$ and $\bS = \Pa(\bX) = \bigcup_{X \in \bX} \Pa(X) \subseteq \ND(\bX) = \bZ$ are the parents of $\bX$.
Then, any subset $\bS' \subseteq \ND(\bX) = \bZ$ such that $T_{\bS', \bx, \by} = T_{\bS, \bx, \by}$ must satisfy both (i) $\bY \ind \bS \setminus \bS' \mid \bX \cup \bS'$ and (ii) $\bX \ind \bS' \setminus \bS \mid \bS$.
\end{mylemma}
\begin{proof}
We know that $\bS = \Pa(\bX)$ is a valid adjustment set and so it must block any non-causal paths between $\bX$ and $\bY$ \citep{perkovic2018complete}.
Then, since $T_{\bS', \bx, \by} = T_{\bS, \bx, \by}$, it must be the case that $\bS'$ is also a valid adjustment set.

\paragraph{Condition (i)}: $\bY \ind \bS \setminus \bS' \mid \bX \cup \bS'$

Suppose, for a contradiction, that $\bY \not\ind \bS \setminus \bS' \mid \bX \cup \bS'$.
By contrapositive of the Markov property (\cref{def:markov}), there is an active d-connected path in $\cG$ from some $Y \in \bY$ to some $A \in \bS \setminus \bS'$, when $\bX \cup \bS'$ is conditioned upon.
Let $\bP_{Y,A}$ denote such an active path of minimal length.
By minimality of $\bP_{Y,A}$, there are no internal vertices from $\bS \setminus \bS'$ within the path $\bP_{Y,A}$.
We will argue that such a path $\bP_{Y,A}$ \emph{cannot} exist by considering the two cases of whether the path $\bP_{Y,A}$ contains some vertex from $\bX$ internally.
\\

\emph{Case 1}: Suppose $\bP_{Y,A}$ contains some vertex from $\bX$, i.e.\ $\bV(\bP_{Y,A}) \cap \bX \neq \emptyset$.
Let $X \in \bX$ be the vertex in $\bV(\bP_{Y,A}) \cap \bX$ that is closest to $Y$, i.e.\ there are no other vertices between $X$ and $Y$ along the path $\bP_{Y,A}$.
Let $\bQ_{Y,X}$ denote this subpath of $\bP_{Y,A}$.
Since $\bP_{Y,A}$ is active with respect to $\bX \cup \bS'$, $X$ must appear as a collider on $\bP_{Y,X}$.
That is, $\bQ_{Y,X}$ is a non-causal path from $X$ to $Y$ that does not contain any internal $\bX$ vertices.
\\

\emph{Case 2}: Suppose $\bP_{Y,A}$ does \emph{not} contain any vertex from $\bX$, i.e.\ $\bV(\bP_{Y,A}) \cap \bX = \emptyset$.
Since $A \in \bS \setminus \bS' \subseteq \bS = \Pa(\bX)$, there must be an edge $A \to X$ for some $X \in \bX$.
Therefore, the extended path $\bQ_{Y,X} = \bP_{Y,A} \cup \{A \to X\}$ is a non-causal path from $X$ to $Y$ that does not contain any internal $\bX$ vertices.
\\

In either case, we have some non-causal path from $X$ to $Y$ that does not contain any internal $\bX$ vertices denoted by $\bQ_{Y,X}$.
Since $\bS'$ is a valid adjustment set, $\bS'$ must block $\bQ_{Y,X}$, which implies that $\bP_{Y,A}$ will be blocked by $\bX \cup \bS'$.
This is a contradiction to the existence of such an active path $\bP_{Y,A}$ in the first place.

\paragraph{Condition (ii)}: $\bX \ind \bS' \setminus \bS \mid \bS$

Suppose, for a contradiction, that $\bX \not\ind \bS' \setminus \bS \mid \bS$.
By contrapositive of the Markov property (\cref{def:markov}), there is an active d-connected path from some $X \in \bX$ to some $B \in \bS' \setminus \bS$, when $\bS$ is being conditioned upon.
Let $\bP_{X,B}$ denote such an active path.
Note that $\bP_{X,B}$ \emph{cannot} begin with an incoming edge into $X$.
This is because otherwise $\bP_{X,B}$ has the form $X \gets C - \ldots$ for some $C \in \Pa(\bX) = \bS$ and so would be not be active when $\bS$ is being conditioned upon.
So, it must be the case that $\bP_{X,B}$ begins with an outgoing edge from $X$.
Then, there must be a collider on $\bP_{X,B}$ involving a descendant of $X$ because $B \in \bS' \setminus \bS \subseteq \ND(\bX)$.
However, the conditioning set $\bS \subseteq \ND(\bX)$ would not include this descendant, so $\bP_{X,B}$ would not be active.
This contradicts the existence of such an active path $\bP_{X,B}$ in the first place.
\end{proof}

\subsection{Valid adjustment by non-descendants in ADMGs}
\label{sec:appendix-ADMG}

Consider an ADMG with single treatment variable $X$ and single outcome variable $Y$, and let $\bZ = \ND(X)$.
Unfortunately, $\bZ$ is not necessarily a valid adjustment set for $\bbP_\bx(\by)$, see \cref{fig:nd-not-valid} for a counterexample.
To ensure that $\bZ$ is a valid adjustment set for $\bbP_\bx(\by)$, we must introduce an additional assumption on the causal graph $\cG$.
In particular, let $\Dis(X)$ denote the \textit{district} (a.k.a.\ \textit{c-component}) of $X$, i.e.\ the set of all nodes in $\cG$ that are connected to $X$ by only bidirected edges.
Following Definition 17 of \cite{richardson2023nested}, we say that $X$ is \textit{fixable} if $\Dis(X) \cap \De(X) = \{ X \}$, i.e.\ if the district of $X$ does not contain any of its descendants.
Further, we say that a path from $X$ to $Y$ is \textit{causal} if it is of the form $X \to \ldots \to Y$, i.e.\ if it contains only directed edges pointing away from $X$; otherwise, we call a path \textit{non-causal}.
Then, we have the following:

\begin{myassumption}\label{assumption:fixable}
    Assume that $X$ is fixable in $\cG$ (i.e.\ $\Dis(X) \cap \De(X) = \varnothing$) and $Y \not\in \Dis(X)$.
\end{myassumption}

\begin{mylemma}
    Under \cref{assumption:fixable}, $\bZ = \ND(X)$ is a valid adjustment set for $\bbP_\bx(\by)$.
\end{mylemma}
\begin{proof}
    \cite{shpitser2010validity} showed that a set $\bZ$ is a valid adjustment set for $\bbP_\bx(\by)$ if it satisfies the following \textit{adjustment criterion}:
    \begin{itemize}
        \item[(i)] No $Z \in \bZ$ is a descendant in $\cG_{\overline{\bX}}$ of any $W$ on a causal path from $X$ to $Y$.

        \item [(ii)] All non-causal paths from $X$ to $Y$ are blocked by $\bZ$.
    \end{itemize}

    Any $Z$ violating Condition (i) must be a descendant of $X$; thus, Condition (i) is immediately satisfied for $\bZ = \ND(X)$.

    To see (ii), for any non-causal path from $X$ to $Y$, let $A$ be the node closest to $X$ on that path which is not a collider; such a node must exist since we assume that $Y \not\in \Dis(X)$.
    Then either $A \in \Dis(X)$ or $A \in \Pa(\Dis(X))$.
    Since $X$ is fixable, both of these cases imply that $A \in \ND(X)$. Thus, any non-causal path is blocked by $\bZ = \ND(X)$.
\end{proof}

\begin{figure}[htb]
    \centering
    \begin{tikzpicture}
        \node[draw, thick, circle] at (-1,-1) (x) {$X$};
        \node[draw, thick, circle] at (0,0) (a) {$A$};
        \node[draw, thick, circle] at (1,-1) (b) {$B$};
        \node[draw, thick, circle] at (2.5,-1) (y) {$Y$};
        \draw[thick, Stealth-Stealth] (a) -- (x);
        \draw[thick, Stealth-Stealth] (a) -- (b);
        \draw[thick, -Stealth] (x) -- (b);
        \draw[thick, -Stealth] (b) -- (y);
    \end{tikzpicture}
    \caption{An ADMG in which $\ND(X) = \{ A \}$ is not a valid adjustment set for $\bbP_\bx(\by)$. In particular, when conditioning on $A$, there is an m-connecting non-causal path $X \leftrightarrow A \leftrightarrow B \to Y$. Note that $\Dis(X) \cap \De(X) = \{ B \}$, i.e., $X$ is not a fixable node and thus does not satisfy \cref{assumption:fixable}.}
    \label{fig:nd-not-valid}
\end{figure}
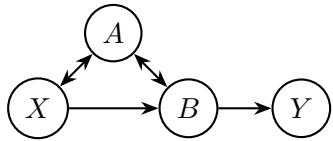




Note that this result can also be obtained more directly using the conditional ADMG framework of \cite{richardson2023nested}; we have chosen to give this slightly longer proof to minimize the required background.

This result lets us extend the first part of our interpretation from DAGs to ADMGs: under \cref{assumption:fixable}, $\bZ = \ND(X)$ is a valid adjustment set for $\bbP_\bx(\by)$, and hence a valid starting point for our algorithm.
To generalize the second part of our interpretation, we note that $\ND(X)$ is an \textit{ancestral set}, i.e.\ for any $Z \in \ND(\bX)$, $\An(Z) \subseteq \ND(X)$.
From this fact, we rather directly have the following:

\begin{mylemma}
    Let $\bZ = \ND(X)$ and $\bS = \Pa(\Dis(X)) \cup \Dis(X) \setminus \{ X \}$.
    Under \cref{assumption:fixable}, $\bS \subseteq \bZ$ and $\bS$ is a Markov blanket of $X$ with respect to $\bZ$.
\end{mylemma}
\begin{proof}
    The fact that $\bS \subseteq \bZ$ follows from directly from the fixability condition.
    Applying the \textit{ordered local Markov property} \cite[Section 3]{richardson2003markov} to the ancestral set $\ND(X)$, we obtain that $\bS$ is a Markov blanket of $X$ with respect to $\bZ$.
\end{proof}

\subsection{Weak edges}
\label{sec:appendix-weak-edges}

In this section, we describe a simple concrete example whereby it is suboptimal to first learn a correct causal graph and then apply identifiability formulas to estimate $\bbP_{\bx}(\by)$.
In particular, correctly learning the causal graph $\cG^*$ may require taking a large number of samples, especially in the presence of ``weak edges''.
However, one would expect such edges to contribute little to $\bbP_{\bx}(\by)$.

Suppose we have a probability distribution $\bbP$ on variables $\{X,Y,Z\}$ generated as follows:
\begin{align*}
Z \gets & \Bern(1/2)\\
X \gets & \begin{cases}
Z & \text{with probability $\eps > 0$}\\
\Bern(1/2) & \text{with probability $1-\eps$}
\end{cases}\\
Y \gets & X \xor Z
\end{align*}
The causal graph that exactly captures $\bbP$ is a complete DAG with edges $Z \to X \to Y$ and $Z \to Y$; see $\cG_1$ in \cref{fig:weak-edge-example}.
However, for extremely small $\eps$, one would require $\Omega(1/\eps)$ samples to detect a dependency between $X$ and $Z$.
So, with small $\eps$ and insufficient samples, one may erroneously recover a subgraph without the $Z \to X$ arc; see $\cG_2$ in \cref{fig:weak-edge-example}.

\begin{figure}[htb]
    \centering
    \resizebox{0.6\linewidth}{!}{
    \begin{tikzpicture}
\node[] at (-2,-0.5) {$\cG_1$};
\node[draw, thick, circle] at (0,0) (z1) {$Z$};
\node[draw, thick, circle] at (-1,-1) (x1) {$X$};
\node[draw, thick, circle] at (1,-1) (y1) {$Y$};
\draw[thick, -Stealth] (z1) -- (x1);
\draw[thick, -Stealth] (z1) -- (y1);
\draw[thick, -Stealth] (x1) -- (y1);

\draw[thick, dashed] (2,0.5) -- (2,-1.5);

\node[] at (6,-0.5) {$\cG_2$};
\node[draw, thick, circle] at (4,0) (z2) {$Z$};
\node[draw, thick, circle] at (3,-1) (x2) {$X$};
\node[draw, thick, circle] at (5,-1) (y2) {$Y$};
\draw[thick, -Stealth] (z2) -- (y2);
\draw[thick, -Stealth] (x2) -- (y2);
\end{tikzpicture}
    }
    \caption{While it is hard to distinguish $\cG_1$ from $\cG_2$ for small $\eps$ with few samples from $\bbP$, estimating $\bbP_{x}(y)$ using $\cG_2$ only incurs an additive error of $O(\eps)$.}
    \label{fig:weak-edge-example}
\end{figure}
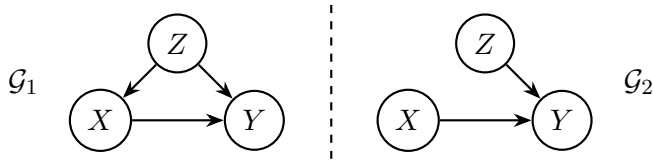

Now, suppose we are interested in estimating $\bbP_0(1) = \bbP(Y = 1 \mid \DO(X = 0))$ from observational data.
One can check that the correct answer is $\bbP(Y = 1 \mid \DO(X = 0)) = 1/2$.
Applying standard adjustment formulas under $\cG_1$ yield $\bbP(Y = 1 \mid \DO(X = 0)) = \sum_{z \in \{0,1\}} \bbP(Z = z) \cdot \bbP(Y = 1 \mid X = 0, Z = z) = 1/2$ as expected.
Meanwhile, under $\cG_2$, the estimation would simply by $\bbP(Y = 1 \mid X = 0) = (1-\eps)/2 = 1/2 - \eps/2$.
Thus, see that the estimation error is only an additive $O(\eps)$ factor away from the ground truth.

\subsection{Stronger results under causal faithfulness}
\label{sec:appendix-faithfulness}

Recall from \AMBAshort\ (\cref{alg:AMBA}) that we need to perform conditional independence checks of the form $\bX \ind_\eps \ND(\bX) \setminus \bS \mid \bS$, which could potentially involve up to $|\bV|$ variables.
Furthermore, we also know from \cref{thm:AMBA} that the required sample complexity of conditional independence testing typically increases as the total number of variables involved increases.
Thus, it would be preferable if we just check whether $\bX \ind_\eps V \mid \bS$ for each $V \in \ND(\bX) \setminus \bS$, and derive that $\bX \ind \ND(\bX) \mid \bS$.
When $\eps = 0$, this implication is a form \textit{compositionality}, and is well-known to hold under the faithfulness assumption (since the set of d-separation statements in a graph is a \textit{graphoid}, see e.g. \cite[Chapter 1]{maathuis2018handbook}), we provide an elementary proof below.

\begin{mylemma}
\label{lem:can-merge-independencies}
Let $\bA, \bB, \bC, \bD$ be disjoint subsets of variables.
Under the causal faithfulness assumption, if $\bA \ind \bB \mid \bC$ and $\bA \ind \bD \mid \bC$, then $\bA \ind (\bB \cup \bD) \mid \bC$.
\end{mylemma}
\begin{proof}
Suppose, for a contradiction, that $\bA \not\ind (\bB \cup \bD) \mid \bC$.
Under the causal faithfulness assumption, this means that there is a d-connected path $P$ from some $A \in \bA$ to some $V \in \bB \cup \bD$ that is active with respect to $\bC$.
Without loss of generality, due to symmetry of the statement, suppose that $V \in \bB$.
That is, $P$ is a path from $A \in \bA$ to some $V \in \bV$ that is active with respect to $\bC$.
But such an active path $P$ contradicts the assumption that $\bA \ind \bB \mid \bC$.
Contradiction.
\end{proof}

Note that \cref{lem:can-merge-independencies} is \emph{false} in general with respect to unfaithful distributions.

\begin{example}\label{example:violation-of-faithfulness}
The simple 3-variable distribution $X = Z_1 \oplus Z_2$, where $Z_1$ and $Z_2$ are independent fair coin flips is unfaithful to any DAG on 3 nodes.
To see why, observe that any two variables are unconditionally independent but completely dependent upon conditioning on the third.
So, one would minimally have to use a v-structure, say $Z_1 \to X \gets X_2$ to represent this.
However, $Z_1 \to X$ is an active path which implies $Z_1 \not\ind X$ under the causal faithfulness assumption, which is not true in $\bbP(Z_1, Z_2, X)$.
\end{example}

Unfortunately, faithfulness alone is not sufficient to ensure the desired implication.
As we demonstrate in the following example (a minor adaptation of the above example), a distribution $\bbP(\bV)$ may be faithful to a DAG, but fail to satisfy the desired compositionality-style property.



\begin{mylemma}
Let $0 < \eps \leq 1/2$.
Consider a probability distribution $\bbP$ over three binary variables $(A, B, X)$ where $A \in \Bern(1/2)$ and $B \in \Bern(1/2)$ are two independent Bernoulli random variables, each with success probability $1/2$ and $X$ is defined as follows:
\[
X =
\begin{cases}
A \oplus B & \text{with probability $1 - 2 \eps$}\\
A & \text{with probability $\eps$}\\
B & \text{with probability $\eps$}
\end{cases}
\]
We have $\Delta_{X \ind A \mid \varnothing} = \Delta_{X \ind B \mid \varnothing} = \eps$ and $\Delta_{X \ind (A,B) \mid \varnothing} = \frac{1}{2} - \eps$.
\end{mylemma}
\begin{proof}
By construction, $\bbP(A = 0) = \bbP(B = 0) = \bbP(X = 0) = 1/2$.
Meanwhile, one can check that $\bbP(X = 0, A = 0) = \bbP(X = 1, A = 1) = \frac{1}{4} + \frac{\eps}{2}$ and $\bbP(X = 0, A = 1) = \bbP(X = 1, A = 0) = \frac{1}{4}$.
For instance, $\bbP(X = 0, A = 0) = {\color{blue}\bbP(X = 0 \mid A = 0)} \cdot {\color{red}\bbP(A = 0)} = {\color{blue}\left( (1 - \eps) \cdot \frac{1}{2} + \eps + \eps \cdot \frac{1}{2} \right)} \cdot {\color{red}\frac{1}{2}} = \frac{1}{4} + \frac{\eps}{2}$.
So, $\sum_{x,a \in \{0,1\}} |\bbP(x,a) - \bbP(x) \cdot \bbP(a)| = \eps$.
By \cref{def:approx-cond-ind}, this establishes $\Delta_{X \ind A \mid \varnothing} = \eps$.

The analysis of $\Delta_{X \ind B \mid \varnothing} = \eps$ is symmetric by replacing the role of $A$ by $B$ in the above analysis.

Since $A$ and $B$ are independent Bernoulli random variables, we see that $\bbP(A = a, B = b) = \bbP(A = a) \cdot \bbP(B = b) = \frac{1}{2} \cdot \frac{1}{2} = \frac{1}{4}$ for any $a,b \in \{0,1\}$.
Meanwhile,
\begin{align*}
\bbP(X = 0 \mid A = 0, B = 0) &= 1\\
\bbP(X = 0 \mid A = 0, B = 1) &= \eps\\
\bbP(X = 0 \mid A = 1, B = 0) &= \eps\\
\bbP(X = 0 \mid A = 1, B = 1) &= 1-2\eps
\end{align*}
So,
\begin{align*}
\sum_{x,a,b \in \{0,1\}} \left| \bbP(x,a,b) - \bbP(x) \cdot \bbP(a,b) \right|
&= \sum_{x,a,b \in \{0,1\}} \bbP(a,b) \cdot \left| \bbP(x \mid a,b) - \bbP(x) \right|\\
&= \frac{1}{4} \cdot \sum_{x,a,b \in \{0,1\}} \left| \bbP(x \mid a,b) - \frac{1}{2} \right| \tag{Since $\bbP(a,b) = \frac{1}{4}$ and $\bbP(x) = \frac{1}{2}$ always}\\
&= \frac{1}{4} \cdot \left( \left| 1 - \frac{1}{2} \right| + \left| \eps - \frac{1}{2} \right| + \left| \eps - \frac{1}{2} \right| + \left| 1 - 2 \eps - \frac{1}{2} \right| \right) \tag{From above}\\
&= \frac{1}{4} \cdot \left( 2 - 4 \eps \right) \tag{Since $\eps \leq \frac{1}{2}$}\\
&= \frac{1}{2} - \eps
\end{align*}
By \cref{def:approx-cond-ind}, this establishes $\Delta_{X \ind (A,B) \mid \varnothing} = \frac{1}{2} - \eps$.
\end{proof}

The above example demonstrates that the faithfulness assumption is insufficient for our purposes.
Instead, we need an assumption of the following form; as we will show, this assumption is implied by a type of \textit{strong faithfulness} assumption.

\begin{mydefinition}
    We say that $\bbP(\bV)$ obeys \emph{$(\eps,\gamma)$-strong compositionality} if, for any disjoint sets $\bA, \bB, \bC, \bD \subseteq \bV$, the following is true:
    \[
    (\bA \ind_\eps \bB \mid \bC)
    \wedge
    (\bA \ind_\eps \bD \mid \bC)
    \implies
    (\bA \ind_{\gamma \eps} \bB \cup \bD \mid \bC)
    \]
\end{mydefinition}

Under $(\eps,\gamma)$-strong compositionality, we can derive that $\bX \ind_\eps \bZ \setminus \bS \mid \bS$ from smaller conditional independence tests; in particular, using a bisection arguments, if $\bX \ind_\eps V \mid \bS$ for all $V \in \bZ \setminus \bS$, then $\bX \ind_{\gamma^k \eps} \bZ \setminus \bS \mid \bS$ for $k = \lceil \log_2(|\bZ \setminus \bS|) \rceil$.
Finally, we relate can strong compositionality to the faithfulness assumption: strong faithfulness implies strong compositionality with $\gamma = 0$, as follows.

\begin{myassumption}[TV Strong faithfulness]
    If $\bA$ is d-connected to $\bB$ given $\bC$ in $\cG^*$, then
    \[
    \Delta_{\bA \ind \bB \mid \bC} > \beta
    \]
    Equivalently, $\Delta_{\bA \ind \bB \mid \bC} \leq \beta \implies \bA$ is d-separated from $\bB$ given $\bC$.
\end{myassumption}

\begin{mylemma}[TV strong faithfulness implies strong compositionality]
    Suppose $\bbP(\bV)$ is $\beta$-TV strong faithful to $\cG^*$.
    Then $\bbP(\bV)$ is $(\beta, 0)$-compositional.
\end{mylemma}
\begin{proof}
    Suppose $\bA \ind_\beta \bB \mid \bC$ and $\bA \ind_\beta \bD \mid \bC$.
    Then, by $\beta$-TV strong faithfulness, $\bA$ is is d-separated from $\bB$ given $\bC$, and $\bA$ is d-separated from $\bD$ given $\bC$.
    Thus, $\bA$ is d-separated from $\bB \cup \bD \mid \bC$, so $\bA \ind \bB \cup \bD \mid \bC$.
\end{proof}


\section{Deferred proof details}
\label{sec:appendix-deferred-proofs}

\subsection{Covariate adjustment in the potential outcomes framework}
\label{sec:appendix-potential-outcomes-adjustment}

For simplicity, we describe the potential outcomes framework in the i.i.d.\ setting, i.e., we assume that $n$ samples are drawn independently from a distribution $\bbP(\bV)$, though we note that the following result can be extended to weaker settings (e.g. when samples are exchangeable but not necessarily independent).

In the PO framework, the treatment variables $\bX$ are considered to be given, along with a set $\bSigma_\bX$ of possible values for $\bX$.
Given $\bX$ and $\bSigma_\bX$, one takes as their starting point an indexed set of random variables $\{ \bY(\bx) \}_{\bx \in \bSigma_\bX}$, with $\bY(\bx)$ denoting the potential outcome associated with intervening to set $\bX$ equal to $\bx$.
Then, the \textit{factual} outcome $\bY$ is generated according to $\bX$ and the potential outcomes; typically, one assumes \textit{consistency}, i.e., that if $\bX = \bx$, then $\bY = \bY(\bx)$.
Hence, under the PO framework, we have $\bbP_\bx(\by) = \bbP(\bY(\bx) = \by)$ is the probability that $\bY$ takes on value $\by$ if $\bX$ is set to $\bx$.

Now, \cref{eqn:covariate-adjustment} can be derived as a consequence of consistency and an additional assumption about conditional independences.
In particular, $\bX$ is called \textit{conditionally ignorable} with respect to $\bZ$ if $\bY(\bx) \ind \bX \mid \bZ$ for all $\bx \in \Sigma_\bX$.



\begin{mylemma}
\label{lem:Z-is-valid-under-SUTVA-and-conditional-ignorability}
Under consistency and conditional ignorability of $\bX$ with respect to $\bZ$, we have
\[
\bbP(\bY(\bx) = \by) = \sum_{\bz \in \bZ} \bbP(\bY = \by \mid \bZ = \bz, \bX = \bx) \cdot \bbP(\bZ = \bz)
\]
\end{mylemma}
\begin{proof}
\begin{align*}
\bbP(\bY(\bx) = y)
&= \sum_{\bz \in \bZ} \bbP(\bY(\bx) = \by \mid \bZ = \bz) \cdot \bbP(\bZ = \bz) \tag{Law of total probability}\\
&= \sum_{\bz \in \bZ} \bbP(\bY(\bx) = \by \mid \bZ = \bz, \bX = \bx) \cdot \bbP(\bZ = \bz) \tag{Since $\bY(\bx) \ind \bX \mid \bZ$}\\
&= \sum_{\bz \in \bZ} \bbP(\bY = \by \mid \bZ = \bz, \bX = \bx) \cdot \bbP(\bZ = \bz) \tag{By consistency}
\end{align*}
\end{proof}

\subsection{Sample complexity for empirical estimation}
\label{sec:appendix-sample-complexity}

The proof of \cref{thm:estimation-error} relies on the following lemma.

\begin{mylemma}
\label{lem:properties-from-poissonization-sampling}
Suppose we have i.i.d.\ sample access to $\bbP(\bV)$.
Given integer $n > 0$ as a sampling parameter, we take $N_{\Pois} \sim \Pois(n)$ samples.
For any $\bU \subseteq \bV$, let the random variable $N_{\bu}$ denote the number of times $\bu \in \bSigma_{\bU}$ was realized within the $n_{\Pois}$ samples.
Then, the following statements hold:
\begin{enumerate}
    \item Let $\bA, \bB \subseteq \bV$ be disjoint sets of variables.
    For any $\ba, \ba' \in \bSigma_{\bA}$ and $\bb, \bb' \in \bSigma_{\bB}$ with $\bb \neq \bb'$, the ratios of random variables $\frac{N_{\ba, \bb}}{N_{\bb}}$ and $\frac{N_{\ba', \bb'}}{N_{\bb'}}$ are independent.
    \item Let $\bA, \bB \subseteq \bV$ be disjoint sets of variables.
    For any $\ba \in \bSigma_{\bA}$, $\bb \in \bSigma_{\bB}$, and integer $k \geq 1$, we have $\left( \frac{N_{\ba, \bb}}{N_{\bb}} - \bbP(\ba \mid \bb) \mid N_{\bb} \geq k \right) \sim \subG\left( \frac{1}{4k} \right)$.
\end{enumerate}
\end{mylemma}
\begin{proof}
We prove each item one at a time.
\begin{enumerate}
    \item By \cref{lem:poissonization}, the random variables $N_{\bb}$ and $N_{\bb'}$ are independent since $\bb \neq \bb'$.
    Then since $N_{\ba,\bb}$ and $N_{\ba',\bb'}$ are subcounts of $N_{\bb}$ and $N_{\bb'}$ respectively, so the corresponding ratios are also independent.
    \item By \cref{lem:poissonization}, we have $( N_{\ba, \bb} \mid N_{\bb} = k ) \sim \Bin(k, \bbP(\ba \mid \bb))$.
    Conditioned on $N_{\bb} = k$, \cref{lem:hoeffding} implies that
    $
    ( N_{\ba, \bb} - \bbE(N_{\ba, \bb}) ) = ( N_{\ba, \bb} - k \cdot \bbP(\ba \mid \bb) ) \sim \subG(\frac{k}{4})
    $.
    Thus, $\left( \frac{N_{\ba, \bb}}{N_{\bb}} - \bbP(\ba \mid \bb) \mid N_{\bb} = k \right) \sim \subG(\frac{1}{4k})$.
    The claim follows via \cref{lem:subgaussian-take-max}.
\end{enumerate}
\end{proof}

\estimationerror*
\begin{proof}[Proof sketch]
By definition, we have
\begin{align*}
T_{\bA,\bx,\by} - \hatT_{\bA,\bx,\by}
&= \sum_\ba \left( \bbP(\ba) \cdot \bbP(\by \mid \bx, \ba) - \frac{N_\ba}{N_\Pois} \cdot \frac{N_{\by,\bx,\ba}}{N_{\bx,\ba}} \right)\\
&= \sum_\ba \bbP(\ba) \cdot \left( \bbP(\by \mid \bx, \ba) - \frac{N_{\by,\bx,\ba}}{N_{\bx,\ba}} \right) + \sum_\ba \left( \bbP(\ba) - \frac{N_\ba}{N_\Pois} \right) \cdot \frac{N_{\by,\bx,\ba}}{N_{\bx,\ba}}
\end{align*}
where $N_\ba$, $N_\Pois$, $N_{\by,\bx,\ba}$, and $N_{\bx,\ba}$ are random Poisson variables from the Poissonization process with $N_\Pois \sim \Pois(n)$ for some parameter $n$; see \cref{sec:poissonization}.
Since $N_\Pois = \sum_{\ba} N_\ba = \sum_{\ba, \bx} N_{\bx,\ba} = \sum_{\ba, \bx, \by} N_{\by,\bx,\ba}$, we see that $0 \leq \frac{N_\ba}{N_\Pois} \leq 1$ and $0 \leq \frac{N_{\by,\bx,\ba}}{N_{\bx,\ba}} \leq 1$ for each of these fractional terms.

Let us define a threshold $\tau > 0$ and partition the values of $\bA$ accordingly:
\[
\bSigma_{\bA \geq \tau} = \left\{ \ba \in \bSigma_{\bA} : \bbP(\bx, \ba) \geq \tau \right\}
\]
Since $\alpha_{\bA} = \min_{\ba \in \bSigma_{\bA}} \bbP(\bx \mid \ba)$, we see that $\bbP(\ba) \leq \frac{\tau}{\alpha_{\bA}}$ for $\ba \not\in \bSigma_{\bA \geq \tau}$.

Let us define three summations $J_{< \tau}$, $J_{\geq \tau}$, and $K$ so that $T_{\bA,\bx,\by} - \hatT_{\bA,\bx,\by} = J_{< \tau} + J_{\geq \tau} + K$:
\begin{align}
J_{< \tau} &= \sum_{\ba \not\in \bSigma_{\bA \geq \tau}} \bbP(\ba) \cdot \left( \bbP(\by \mid \bx, \ba) - \frac{N_{\by,\bx,\ba}}{N_{\bx,\ba}} \right) \label{eq:J-less-tau-defn}\\
J_{\geq \tau} &= \sum_{\ba \in \bSigma_{\bA \geq \tau}} \bbP(\ba) \cdot \left( \bbP(\by \mid \bx, \ba) - \frac{N_{\by,\bx,\ba}}{N_{\bx,\ba}} \right) \label{eq:J-geq-tau-defn}\\
K &= \sum_\ba \left( \bbP(\ba) - \frac{N_\ba}{N_\Pois} \right) \cdot \frac{N_{\by,\bx,\ba}}{N_{\bx,\ba}} \label{eq:K-defn}
\end{align}
We will proceed to bound each of $|J_{< \tau}|$, $|J_{\geq \tau}|$, and $|K|$.

The easiest is $|J_{< \tau}|$, which follows from the definition of $\bSigma_{\bA \geq \tau}$:
\begin{align*}
|J_{< \tau}|
&= \left| \sum_{\ba \not\in \bSigma_{\bA \geq \tau}} \bbP(\ba) \cdot \left( \frac{N_{\by, \bx, \ba}}{N_{\bx, \ba}} - \bbP(\by \mid \bx, \ba) \right) \right| \tag{Definition of $|J_{< \tau}|$}\\
& \leq \sum_{\ba \not\in \bSigma_{\bA \geq \tau}} \bbP(\ba) \cdot \left|  \frac{N_{\by, \bx, \ba}}{N_{\bx, \ba}} - \bbP(\by \mid \bx, \ba) \right| \tag{By triangle inequality and $\bbP(\ba) \geq 0$}\\
&\leq \sum_{\ba \not\in \bSigma_{\bA \geq \tau}} \bbP(\ba) \tag{Since $\left|  \frac{N_{\by, \bx, \ba}}{N_{\bx, \ba}} - \bbP(\by \mid \bx, \ba) \right| \leq 1$}\\
&\leq \frac{\tau \cdot |\bSigma_{\bA}|}{\alpha_{\bA}} \tag{Since $\bbP(\ba) \leq \frac{\tau}{\alpha_{\bA}}$ for $\ba \not\in \bSigma_{\bA \geq \tau}$ and $|\bSigma_{\bA \geq \tau}| \leq |\bSigma_{\bA}|$}
\end{align*}

To bound $|J_{\geq \tau}|$, consider the concentration event $\cE^{J}_{\geq \tau}$ defined as follows:
\begin{equation}
\label{eq:J-event}
\cE^{J}_{\geq \tau}
= \bigcap_{\ba \in \bSigma_{\bA \geq \tau}} \left\{ N_{\bx, \ba} > \frac{n \cdot \bbP(\bx, \ba)}{2} \right\}
\end{equation}

We first observe that the event $\cE^{J}_{\geq \tau}$ holds with good probability.
\begin{align*}
1 - \Pr(\cE^{J}_{\geq \tau})
&\leq \sum_{\ba \in \bSigma_{\bA \geq \tau}} \Pr \left( N_{\bx, \ba} \leq \frac{n \cdot \bbP(\bx, \ba)}{2} \right) \tag{Union bound}\\
&\leq \sum_{\ba \in \bSigma_{\bA \geq \tau}} \exp \left( -\frac{n \cdot \bbP(\bx, \ba)}{12} \right) \tag{Using that $N_{\bx, \ba} \sim \Pois(n \cdot \bbP(\bx, \ba))$ and applying \cref{lem:poisson-concentration}}\\
&\leq |\bSigma_{\bA}| \cdot \exp \left( -\frac{n \tau}{12} \right) \tag{Since $\bbP(\bx, \ba) \geq \tau$ for $\bz \in \bSigma_{\bA \geq \tau} \subseteq \bSigma_{\bA \geq \tau}$}
\end{align*}

Under the event $\cE^{J}_{\geq \tau}$, we have $N_{\bx, \ba} > \frac{n \cdot \bbP(\bx, \ba)}{2}$ for any $\ba \in \bSigma_{\bA \geq \tau}$, so item 2 of \cref{lem:properties-from-poissonization-sampling} implies that
\[
\left( \frac{N_{\by, \bx, \ba}}{N_{\bx, \ba}} - \bbP(\by \mid \bx, \ba) \;\middle|\; N_{\bx, \ba} \geq \frac{n \cdot \bbP(\bx, \ba)}{2} \right)
\sim \subG \left( \frac{1}{4} \cdot \frac{2}{n \cdot \bbP(\bx, \ba)} \right)
= \subG \left( \frac{1}{2 n \cdot \bbP(\bx, \ba)} \right) \;,
\]
for any $\ba \in \bSigma_{\bA \geq \tau}$.
For any two disjoint $\ba, \ba' \in \bSigma_{\bA}$, we see that $(\bx, \ba)$ and $(\bx, \ba')$ are distinct values in the domain $\bSigma_{\bX} \times \bSigma_{\bA}$, so item 1 of \cref{lem:properties-from-poissonization-sampling} tells us that the terms $\frac{N_{\by, \bx, \ba}}{N_{\bx, \ba}}$ and $\frac{N_{\by, \bx, \ba'}}{N_{\bx, \ba'}}$ are independent.
\cref{lem:subgaussian-sum} further tells us that
$
J_{\geq \tau} = \sum_{\ba \in \bSigma_{\bA \geq \tau}} \bbP(\ba) \cdot \left( \frac{N_{\by, \bx, \ba}}{N_{\bx, \ba}} - \bbP(\by \mid \bx, \ba) \right)
\sim \subG \left( \sum_{\ba \in \bSigma_{\bA \geq \tau}} \frac{\bbP(\ba)^2}{2 n \cdot \bbP(\bx, \ba)} \right)
$ since coefficients $\{ \bbP(\ba) \}_{\ba \in \bA}$ are just (unknown) real numbers.
Then, for any $t > 0$, \cref{def:subgaussian} states that
\[
\Pr \left( \left| J_{\geq \tau} \right| > t \mid \cE^{J}_{\geq \tau} \right)
\leq 2 \exp \left( -\frac{t^2}{\sum_{\ba \in \bSigma_{\bA \geq \tau}} \frac{\bbP(\ba)^2}{2 n \cdot \bbP(\bx, \ba)}} \right)
\leq 2 \exp \left( - 2 n \alpha_{\bA} t^2 \right)
\]
where the last inequality is because $\alpha_{\bA} = \min_{\ba \in \bSigma_{\bA}} \bbP(\bx \mid \ba)$ and $\bSigma_{\bA \geq \tau} \subseteq \bSigma_{\bA}$:
\[
\sum_{\ba \in \bSigma_{\bA \geq \tau}} \frac{\bbP(\ba)^2}{2 n \cdot \bbP(\bx, \ba)}
= \sum_{\ba \in \bSigma_{\bA \geq \tau}} \frac{\bbP(\ba)}{2 n \cdot \bbP(\bx \mid \ba)}
\leq \sum_{\ba \in \bSigma_{\bA \geq \tau}} \frac{\bbP(\ba)}{2 n \cdot \alpha_{\bA}}
\leq \frac{1}{2 n \cdot \alpha_{\bA}}
\]

To bound $|K|$, we reduce to the analysis to the problem of producing an $\eps$-close estimate of $\bbP(\bA)$ by observing that $0 \leq \frac{N_{\by,\bx,\ba}}{N_{\bx,\ba}} \leq 1$ and $\frac{N_\ba}{N_\Pois}$ is the empirical estimate of $\bbP(\ba)$ for each $\ba \in \bSigma_{\bA}$.
That is,
\begin{align*}
|K|
&= \left| \sum_\ba \left( \bbP(\ba) - \frac{N_\ba}{N_\Pois} \right) \cdot \frac{N_{\by,\bx,\ba}}{N_{\bx,\ba}} \right| \tag{By \cref{eq:K-defn}}\\
&= \sum_\ba \left| \bbP(\ba) - \frac{N_\ba}{N_\Pois} \right| \cdot \left| \frac{N_{\by,\bx,\ba}}{N_{\bx,\ba}} \right| \tag{By triangle inequality}\\
&\leq \sum_\ba \left| \bbP(\ba) - \frac{N_\ba}{N_\Pois} \right| \tag{Since $0 \leq \frac{N_{\by,\bx,\ba}}{N_{\bx,\ba}} \leq 1$}\\
&\leq \sum_\ba \left| \bbP(\ba) - \wh{\bbP}(\ba) \right| \tag{By defining empirical distribution $\wh{\bbP}(\ba) = \frac{N_\ba}{N_\Pois}$}
\end{align*}
By \cref{lem:TV-estimation-guarantees}, when $N_\Pois \geq c_0 \cdot \left( \frac{|\bSigma_{\bA}| + \log \frac{1}{\delta'}}{(\eps')^2} \right)$ for some tolerance parameters $\eps', \delta' > 0$ and absolute constant $c_0 > 0$, we will have
$
\Pr \left( |K| \leq \eps' \right)
\leq \Pr \left( \sum_{\ba \in \bSigma_{\bA}} | \bbP(\ba) - \wh{\bbP}(\ba) | \leq \eps' \right)
\geq 1 - \delta'
$.

Before we proceed to wrap up the proof, let us collect the proven bounds below:
\begin{align}
|J_{< \tau}|
& \leq \frac{\tau \cdot |\bSigma_{\bA}|}{\alpha_{\bA}}
&& \text{deterministically} \label{eq:J-less-tau-bound}\\
\Pr( \lnot \cE^{J}_{\geq \tau})
& \leq |\bSigma_{\bA}| \cdot \exp \left( -\frac{n \tau}{12} \right) \label{eq:J-event-fails}\\
\Pr \left( \left| J_{\geq \tau} \right| > t \mid \cE^{J}_{\geq \tau} \right)
&\leq 2 \exp \left( - 2 n \alpha_{\bA} t^2 \right)
&& \text{for any $t > 0$} \label{eq:J-geq-tau-bound}\\
\Pr \left( |K| \leq \eps' \right)
&\leq 1 - \delta'
&& \text{for any $\eps', \delta' > 0$ when $N_\Pois \in \cO \left( \frac{|\bSigma| + \log \frac{1}{\delta'}}{(\eps')^2} \right)$} \label{eq:K-bound}
\end{align}

Now, observe that $|J_{< \tau}| \leq \frac{\eps}{3}$, $|J_{\geq \tau}| \leq \frac{\eps}{3}$ and $|K| \leq \frac{\eps}{3}$ jointly implies $|J_{< \tau} + J_{\geq \tau} + K| \leq |J_{< \tau}| + |J_{\geq \tau}| + |K| \leq \eps$ by triangle inequality.
So,
\begin{align*}
&\; \Pr \left( \left| T_{\bA,\bx,\by} - \hatT_{\bA,\bx,\by} \right| > \eps \right)\\
= &\; \Pr \left( |J_{< \tau} + J_{\geq \tau} + K| > \eps \right) \tag{By definition}\\
\leq &\; \Pr \left( |J_{< \tau}| > \frac{\eps}{3} \right) + \Pr \left( |J_{\geq \tau}| > \frac{\eps}{3} \right) + \Pr \left( |K| > \frac{\eps}{3} \right) \tag{Triangle inequality}\\
\leq &\; 0 + \Pr \left( |J_{\geq \tau}| > \frac{\eps}{3} \right) + \Pr \left( |K| > \frac{\eps}{3} \right) \tag{If we set $\frac{\eps}{3} = \frac{\tau \cdot |\bSigma_{\bA}|}{\alpha_{\bA}}$ in the deterministic bound of \cref{eq:J-less-tau-bound}}\\
\leq &\; \Pr \left( |J_{\geq \tau}| > \frac{\eps}{3} \right) + {\color{orange}\frac{\delta}{3}} \tag{If we set $\eps' = \frac{\eps}{3}$ and $\delta' = \frac{\delta}{3}$ in \cref{eq:K-bound} with $N_\Pois \in {\color{orange}\cO \left( \frac{|\bSigma_{\bA}| + \log \frac{1}{\delta'}}{(\eps')^2} \right)}$}\\
\leq &\; \Pr( \lnot \cE^{J}_{\geq \tau}) + \Pr \left( |J_{\geq \tau}| > \frac{\eps}{3} \mid \cE^{J}_{\geq \tau} \right) + {\color{orange}\frac{\delta}{3}} \tag{Conditioning on event $\cE^{J}_{\geq \tau}$}\\
\leq &\; {\color{blue}|\bSigma_{\bA}| \cdot \exp \left( -\frac{n \tau}{12} \right)} + {\color{red}2 \exp \left( - 2 n \alpha_{\bA} t^2 \right)} + {\color{orange}\frac{\delta}{3}} \tag{If we set $t = \frac{\eps}{3}$ then apply \cref{eq:J-event-fails} and \cref{eq:J-geq-tau-bound}}
\end{align*}
Recall that we set $\frac{\eps}{3} = \frac{\tau \cdot |\bSigma_{\bA}|}{\alpha_{\bA}} \iff \tau = \frac{\eps \alpha_{\bA}}{|3 \bSigma_{\bA}|}$ and $t = \frac{\eps}{3}$ above.
So, if we set
\begin{align*}
n
&= {\color{blue}\frac{36 |\bSigma_{\bA}|}{\eps \alpha_{\bA}} \log \left( \frac{3 |\bSigma_{\bA}|}{\delta} \right)} + {\color{red}\frac{9}{2 \eps^2 \alpha_{\bA}} \log \left( \frac{6}{\delta} \right)} + {\color{orange}\cO \left( \frac{|\bSigma_{\bA}| + \log \frac{1}{\delta}}{\eps^2} \right)}\\
&\in \wt{\cO} \left( \left( \frac{|\bSigma_{\bA}|}{\eps \alpha_{\bA}} + \frac{1}{\eps^2 \alpha_{\bA}} + \frac{|\bSigma_{\bA}|}{\eps^2} \right) \cdot \log \left( \frac{1}{\delta} \right) \right)
\end{align*}
then $\Pr \left( \left| T_{\bA,\bx,\by} - \hatT_{\bA,\bx,\by} \right| > \eps \right) \leq {\color{blue}\frac{\delta}{3}} + {\color{red}\frac{\delta}{3}} + {\color{orange}\frac{\delta}{3}} = \delta$.
\end{proof}

\subsection{Misspecification errors}
\label{sec:appendix-misspecification-errors}

\misspecificationerror*
\begin{proof}
Since $\bS \subseteq \bA$, we see that
\begin{align*}
\left| T_{\bS,\bx,\by} - T_{\bA,\bx,\by} \right|
& = \left| \sum_{\ba} \bbP(\by \mid \ba, \bx) \cdot \bbP(\ba \setminus \bs \mid \bs, \bx) \cdot \bbP(\bs) - \sum_{\ba} \bbP(\by \mid \ba, \bx) \cdot \bbP(\ba) \right| \tag{By \cref{eq:T-alternative} and $\bS \subseteq \bA$}\\
& = \left| \sum_{\ba} \bbP(\by \mid \ba, \bx) \cdot \bbP(\bs) \cdot \left( \bbP(\ba \setminus \bs \mid \bs, \bx) - \bbP(\ba \setminus \bs \mid \bs) \right) \right| \tag{Pull out common terms}\\
& = \left| \sum_{\ba} \bbP(\by \mid \ba, \bx) \cdot \frac{\bbP(\bs, \bx)}{\bbP(\bx \mid \bs)} \cdot \left( \bbP(\ba \setminus \bs \mid \bs, \bx) - \bbP(\ba \setminus \bs \mid \bs) \right) \right|\\
& \leq \sum_{\ba} \frac{\bbP(\bs, \bx)}{\bbP(\bx \mid \bs)} \cdot \left| \bbP(\ba \setminus \bs \mid \bs, \bx) - \bbP(\ba \setminus \bs \mid \bs) \right| \tag{Triangle inequality, non-negativity of probabilities, and since $\bbP(\by \mid \ba, \bx) \leq 1$}\\
& \leq \frac{1}{\alpha_{\bS}} \cdot \sum_{\ba} \bbP(\bs, \bx) \cdot \left| \bbP(\ba \setminus \bs \mid \bs, \bx) - \bbP(\ba \setminus \bs \mid \bs) \right| \tag{By \cref{eq:alpha-def}}\\
& \leq \frac{\eps}{\alpha_{\bS}} \tag{Since $\bX \ind_\eps \bA \setminus \bS \mid \bS$ and using \cref{eq:approx-cond-ind-alternative}}
\end{align*}
\end{proof}

\misspecificationerrorlowerbound*
\begin{proof}
Consider the following probability distribution $\bbP$ defined over 4 binary variables $\{A, B, X, Y\}$ in a topological ordering of $A \prec B \prec X \prec Y$: see \cref{fig:hardness-alpha}.

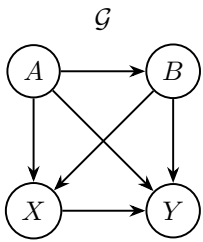
\begin{figure}[htb]
\centering
\resizebox{0.9\linewidth}{!}{
\begin{tikzpicture}
\node[] at (1,0.75) {$\cG$};
\node[draw, thick, circle] at (0,0) (A) {$A$};
\node[draw, thick, circle] at (2,0) (B) {$B$};
\node[draw, thick, circle] at (0,-2) (X) {$X$};
\node[draw, thick, circle] at (2,-2) (Y) {$Y$};
\draw[thick, -Stealth] (A) -- (B);
\draw[thick, -Stealth] (A) -- (X);
\draw[thick, -Stealth] (A) -- (Y);
\draw[thick, -Stealth] (B) -- (X);
\draw[thick, -Stealth] (B) -- (Y);
\draw[thick, -Stealth] (X) -- (Y);

\node[] at (6,-0.75) {\begin{tabular}{l}
$A = \begin{cases}
1 & \text{w.p.\ $\frac{\eps}{4 \alpha} \cdot \frac{\alpha - \eps/4}{1 - \sqrt{\eps}/2}$}\\
0 & \text{else}
\end{cases}$
\\
\\
$B = \begin{cases}
1 - A & \text{w.p.\ $1 - \sqrt{\eps}$}\\
0 & \text{w.p.\ $\sqrt{\eps}/2$}\\
1 & \text{w.p.\ $\sqrt{\eps}/2$}
\end{cases}$
\end{tabular}};

\node[] at (11,-0.75) {\begin{tabular}{l}
$X = \begin{cases}
A & \text{w.p.\ $1 - \alpha$}\\
1 - A & \text{w.p.\ $\alpha - \sqrt{\eps}/2$}\\
B & \text{w.p.\ $\sqrt{\eps}/2$}
\end{cases}$
\\
\\
$Y = \begin{cases}
1 & \text{if $X = 0, A = 1, B = 0$}\\
0 & \text{else}
\end{cases}$
\end{tabular}};
\end{tikzpicture}
}
\caption{Probability distribution $\bbP$ defined over 4 binary variables $\{A, B, X, Y\}$ in a topological ordering of $A \prec B \prec X \prec Y$ with parameters $\eps$ and $\alpha$, where $0 < \sqrt{\eps} \leq \alpha \leq 1/2$.}
\label{fig:hardness-alpha}
\end{figure}

We show in \cref{sec:appendix-hardness} that all the (conditional) probabilities of $\bbP$ are well-defined, and that we have the following conditional probabilities for $\bbP$:
\begin{center}
\begin{tabular}{cc|cccc}
\toprule
$a$ & $b$ & $\bbP(b \mid a)$ & $\bbP(X = 0 \mid a,b)$ & $\bbP(X = 0 \mid a)$ & $\sum_x | \bbP(x \mid a,b) - \bbP(x \mid a)|$\\
\midrule
$0$ & $0$ & $\sqrt{\eps}/2$ & $1 - \alpha + \sqrt{\eps}/2$ & $1 - \alpha + \eps/4$ & $\sqrt{\eps} - \eps/2$\\
$0$ & $1$ & $1-\sqrt{\eps}/2$ & $1 - \alpha$ & $1 - \alpha + \eps/4$ & $\eps/2$\\
$1$ & $0$ & $1-\sqrt{\eps}/2$ & $\alpha$ & $\alpha - \eps/4$ & $\eps/2$\\
$1$ & $1$ & $\sqrt{\eps}/2$ & $\alpha - \sqrt{\eps}/2$ & $\alpha - \eps/4$ & $\sqrt{\eps} - \eps/2$\\
\bottomrule
\end{tabular}
\end{center}

Let us identify $\bZ$ with $\{ A, B \}$ and $\bS$ with $\{A\}$, so $\bZ \setminus \bS = \{B\}$.
We now show the four properties.
\begin{enumerate}
    \item $\bZ$ is a valid adjustment set

    This is true since $\{ A, B \}$ satifies the backdoor adjustment criterion \cite{pearl1995causal}.
    
    \item $\bS \subset \bZ$ satisfies $\bX \ind_\eps \bZ \setminus \bS \mid \bS$

    Recall that $\bZ = \{ A, B \}$ and $\bS = \{A\}$.
    To see that $X \ind_{\eps} B \mid A$, observe the following:
    \begin{align*}
    &\; \sum_{x,a,b} \bbP(a) \cdot |\bbP(x,b \mid a) - \bbP(x \mid a) \cdot \bbP(b \mid a)|\\
    &= \sum_{a,b} \bbP(a) \cdot {\color{red}\bbP(b \mid a)} \cdot {\color{blue}\sum_x |\bbP(x \mid a,b) - \bbP(x \mid a)|}\\
    &= \bbP(A=0) \cdot {\color{red}\bbP(B=0 \mid A=0)} \cdot {\color{blue}(\sqrt{\eps}-\eps/2)}
    + \bbP(A=0) \cdot {\color{red}\bbP(B=1 \mid A=0)} \cdot {\color{blue}(\eps/2)}\\
    &\quad + \bbP(A=1) \cdot {\color{red}\bbP(B=0 \mid A=1)} \cdot {\color{blue}(\eps/2)}
    + \bbP(A=1) \cdot {\color{red}\bbP(B=1 \mid A=1)} \cdot {\color{blue}(\sqrt{\eps}-\eps/2)}\\
    &= \bbP(A=0) \cdot {\color{red}(\sqrt{\eps}/2)} \cdot {\color{blue}(\sqrt{\eps}-\eps/2)}
    + \bbP(A=0) \cdot {\color{red}(1-\sqrt{\eps}/2)} \cdot {\color{blue}(\eps/2)}\\
    &\quad + \bbP(A=1) \cdot {\color{red}(1-\sqrt{\eps}/2)} \cdot {\color{blue}(\eps/2)}
    + \bbP(A=1) \cdot {\color{red}(\sqrt{\eps}/2)} \cdot {\color{blue}(\sqrt{\eps}-\eps/2)}\\
    &= (\sqrt{\eps}/2) \cdot (\sqrt{\eps}-\eps/2) + (1-\sqrt{\eps}/2) \cdot \eps/2\\
    &= \eps
    \end{align*}
    
    \item $\alpha_{\bS} \geq \alpha$

    Since $\bS = \{A\}$ and $\eps \leq \alpha/2$, we have
    $
    \min_{a} \bbP(x \mid a) = \alpha - \eps/4 \geq \alpha/2
    $.
    
    \item $|T_{\bS,\bx,\by} - T_{\bZ,\bx,\by}| \geq \frac{\eps}{16 \alpha}$.

    \begin{align*}
    |T_{\bS,\bx,\by} - T_{\bZ,\bx,\by}|
    &= \left| \sum_a \bbP(a) \cdot \bbP(y \mid x,a) - \sum_{a,b} \bbP(a,b) \cdot \bbP(y \mid x,a,b) \right| \tag{Since $\bS = \{A\}$, $\bZ = \{A,B\}$, and by definition of $T_{\bS,\bx,\by}$ and $T_{\bZ,\bx,\by}$}\\
    &= \left| \sum_{a,b} \bbP(a) \cdot \bbP(y \mid x,a,b) \cdot \left( \bbP(b \mid a) - \bbP(b \mid x,a) \right) \right| \tag{Since $\bbP(y \mid x, a) = \sum_b \bbP(y, b \mid x, a) = \sum_b \bbP(y \mid x, a, b) \cdot \bbP(b \mid x, a)$ and $\bbP(a,b) = \bbP(a) \cdot \bbP(b \mid a)$}\\
    &= \left| \sum_{a,b} \bbP(a) \cdot \bbP(y \mid x,a,b) \cdot \frac{\bbP(b \mid a)}{\bbP(x \mid a)} \cdot \left( \bbP(x \mid a) - \bbP(x \mid a, b) \right) \right| \tag{Since $\bbP(b \mid x, a) = \frac{\bbP(b \mid a) \cdot \bbP(x \mid a, b)}{\bbP(x \mid a)}$}\\
    &= {\color{red}\bbP(A=1)} \cdot {\color{blue}\frac{\bbP(B=0 \mid A=1)}{\bbP(X=0 \mid A=1)}} \cdot \Big| \bbP(X=0 \mid A=1) - \bbP(X=0 \mid A=1, B=0) \Big| \tag{Since $Y$ is an indicator variable for whether $(A,B,X) = (1,0,0)$}\\
    &= {\color{red}\frac{\eps}{4 \alpha} \cdot \frac{\alpha - \eps/4}{1 - \sqrt{\eps}/2}} \cdot {\color{blue}\frac{1 - \sqrt{\eps}/2}{\alpha - \eps/4}} \cdot \frac{\eps}{4} \tag{From construction in \cref{fig:hardness-alpha}}\\
    &= \frac{\eps}{16 \alpha}
    \end{align*}
\end{enumerate}
\end{proof}

\subsection{Beyond approximate Markov blankets}
\label{sec:appendix-BAMBA}

\minimalsoundness*
\begin{proof}
Consider arbitrary subsets $\bS \subseteq \bA \subseteq \bV$ and $\bS' \subseteq \bA \subseteq \bV$.
Observe that
\begin{align*}
T_{\bS, \bx, \by}
&= \sum_{\bs,\bs' \setminus \bs}
\bbP(\by \mid \bx, \bs, \bs' \setminus \bs)
\cdot \bbP(\bs' \setminus \bs \mid \bx, \bs)
\cdot \bbP(\bs)
\tag{By \cref{eq:T-alternative}}
\\
&= \sum_{\bs,\bs' \setminus \bs}
\bbP(\by \mid \bx, \bs, \bs' \setminus \bs)
\cdot \bbP(\bs' \setminus \bs \mid \bs)
\cdot \bbP(\bs)
\tag{Since $\bX \ind \bS' \setminus \bS \mid \bS$}
\\
&= \sum_{\bs',\bs \setminus \bs'}
\bbP(\by \mid \bx, \bs', \bs \setminus \bs')
\cdot \bbP(\bs' \setminus \bs \mid \bs)
\cdot \bbP(\bs)
\tag{Regrouping}
\\
&= \sum_{\bs',\bs \setminus \bs'}
\bbP(\by \mid \bx, \bs')
\cdot \bbP(\bs' \setminus \bs \mid \bs)
\cdot \bbP(\bs)
\tag{Since $\bY \ind \bS \setminus \bS' \mid \bX \cup \bS'$}
\\
&=
T_{\bS',\bx,\by}
\tag{By \cref{eq:T-alternative}}
\end{align*}
\end{proof}

\minimalgivenapproximatemarkovblanketdiscovery*
\begin{proof}
Suppose the \BAMBA\ algorithm (\cref{alg:BAMBA}) terminates at some iteration $|\bS'| \in \{0, 1, \ldots, |\bA|\}$.

\paragraph{Correctness.}
If \BAMBAshort\ returns the $\eps$-Markov blanket $\bS \subseteq \bA$ (e.g.\ in Line 8), then $| T_{\bS', \bx, \by} - T_{\bA, \bx, \by} | = | T_{\bS, \bx, \by} - T_{\bA, \bx, \by} | \leq \frac{\eps}{\alpha_{\bS}} \leq \frac{2 \eps}{\alpha_{\bS}}$ by \cref{def:markov-blanket} and \cref{lem:misspecification-error}.
Suppose all calls to \ApproxCondInd\ succeed across all iterations.
Then, \cref{lem:approxcondind-guarantees} tells us that $\Delta_{\bY \ind \bS \setminus \bS' \mid \bX \cup \bS'} \leq \eps$, $\Delta_{\bX \ind \bS' \setminus \bS \mid \bS} \leq \eps$, and $|\bSigma_{\bS'}| \leq |\bSigma_{\bS}|$ whenever $\bC_k \neq \emptyset$.

For subsequent analytical purposes, let us define an intermediate term $Z_{\bx, \by}$ as follows:
\begin{equation}
\label{eq:intermediate-term}
Z_{\bx, \by}
= \sum_{\bs \cup \bs'} \frac{1}{\bbP(\bx \mid \bs)} \cdot \bbP(\bx, \bs \cup \bs') \cdot \bbP(\by \mid \bx, \bs')
\end{equation}

By triangle inequality, we have
$
\left| T_{\bS, \bx, \by} - T_{\bS', \bx, \by} \right|
= \left| T_{\bS, \bx, \by} - Z_{\bx, \by} + Z_{\bx, \by} - T_{\bS', \bx, \by} \right|
\leq \left| T_{\bS, \bx, \by} - Z_{\bx, \by} \right| + \left| Z_{\bx, \by} - T_{\bS', \bx, \by} \right|
$.
We will bound each of these terms separately.

\textbf{1. Bounding $\left| T_{\bS, \bx, \by} - Z_{\bx, \by} \right|$.}

\begin{align*}
\left| T_{\bS, \bx, \by} - Z_{\bx, \by} \right|
& = \left| \sum_{\bs \cup \bs'} \bbP(\by \mid \bx, \bs \cup \bs') \cdot \bbP(\bs' \setminus \bs \mid \bx, \bs) \cdot \bbP(\bs) - \sum_{\bs \cup \bs'} \frac{1}{\bbP(\bx \mid \bs)} \cdot \bbP(\bx, \bs \cup \bs') \cdot \bbP(\by \mid \bx, \bs') \right| \tag{By \cref{eq:T-alternative} and \cref{eq:intermediate-term}}\\
& = \left| \sum_{\bs \cup \bs'} \bbP(\by \mid \bx, \bs \cup \bs') \cdot \frac{\bbP(\bx, \bs \cup \bs')}{\bbP(\bx, \bs)} \cdot \bbP(\bs) - \sum_{\bs \cup \bs'} \frac{1}{\bbP(\bx \mid \bs)} \cdot \bbP(\bx, \bs \cup \bs') \cdot \bbP(\by \mid \bx, \bs') \right|\\
& = \left| \sum_{\bs \cup \bs'} \frac{1}{\bbP(\bx \mid \bs)} \cdot \bbP(\bx, \bs \cup \bs') \cdot \left( \bbP(\by \mid \bx, \bs \cup \bs') - \bbP(\by \mid \bx, \bs') \right) \right| \tag{Pull out common terms}\\
& \leq \sum_{\bs \cup \bs'} \frac{1}{\bbP(\bx \mid \bs)} \cdot \bbP(\bx, \bs \cup \bs') \cdot \left| \bbP(\by \mid \bx, \bs \cup \bs') - \bbP(\by \mid \bx, \bs') \right| \tag{Triangle inequality and non-negative of probabilities}\\
& \leq \frac{1}{\alpha_{\bS}} \cdot \sum_{\bs \cup \bs'} \bbP(\bx, \bs \cup \bs') \cdot \left| \bbP(\by \mid \bx, \bs \cup \bs') - \bbP(\by \mid \bx, \bs') \right| \tag{By definition of $\alpha_{\bS}$ in \cref{eq:alpha-def}}\\
& \leq \frac{1}{\alpha_{\bS}} \cdot \sum_{\by, \bx, \bs \cup \bs'} \bbP(\bx, \bs \cup \bs') \cdot \left| \bbP(\by \mid \bx, \bs \cup \bs') - \bbP(\by \mid \bx, \bs') \right| \tag{Summing over more terms}\\
& \leq \frac{\eps}{\alpha_{\bS}} \tag{Since $\Delta_{\bY \ind \bS \setminus \bS' \mid \bX \cup \bS'} \leq \eps$ and using \cref{eq:approx-cond-ind-alternative}}
\end{align*}

\textbf{2. Bounding $\left| Z_{\bx, \by} - T_{\bS', \bx, \by} \right|$.}

\begin{align*}
\left| T_{\bS', \bx, \by} - Z_{\bx, \by} \right|
& = \left| \sum_{\bs \cup \bs'} \bbP(\by \mid \bx, \bs') \cdot \bbP(\bs') \cdot \bbP(\bs \setminus \bs' \mid \bs') - \sum_{\bs \cup \bs'} \frac{1}{\bbP(\bx \mid \bs)} \cdot \bbP(\bx, \bs \cup \bs') \cdot \bbP(\by \mid \bx, \bs') \right| \tag{By \cref{eq:T-alternative} and \cref{eq:intermediate-term}}\\
& = \left| \sum_{\bs \cup \bs'} \frac{1}{\bbP(\bx \mid \bs)} \cdot \bbP(\by \mid \bx, \bs') \cdot \bbP(\bs \cup \bs') \cdot \left( \bbP(\bx \mid \bs) - \bbP(\bx \mid \bs \cup \bs') \right) \right| \tag{Pull out common terms}\\
& \leq \sum_{\bs \cup \bs'} \frac{1}{\bbP(\bx \mid \bs)} \cdot \bbP(\bs \cup \bs') \cdot \left| \bbP(\bx \mid \bs) - \bbP(\bx \mid \bs \cup \bs') \right| \tag{Triangle inequality, non-negativity of probabilities, and since $\bbP(\by \mid \bx, \bs') \leq 1$}\\
& \leq \frac{1}{\alpha_{\bS}} \cdot \sum_{\bs \cup \bs'} \bbP(\bs \cup \bs') \cdot \left| \bbP(\bx \mid \bs) - \bbP(\bx \mid \bs \cup \bs') \right| \tag{By definition of $\alpha_{\bS}$ in \cref{eq:alpha-def}}\\
& \leq \frac{\eps}{\alpha_{\bS}} \tag{Since $\Delta_{\bX \ind \bS' \setminus \bS \mid \bS} \leq \eps$ and using \cref{eq:approx-cond-ind-alternative}}
\end{align*}

\textbf{Putting together.}

We see that
\[
\left| T_{\bS, \bx, \by} - T_{\bS', \bx, \by} \right|
\leq \left| T_{\bS, \bx, \by} - Z_{\bx, \by} \right| + \left| Z_{\bx, \by} - T_{\bS', \bx, \by} \right|
\leq \frac{\eps}{\alpha_{\bS}} + \frac{\eps}{\alpha_{\bS}}
= \frac{2 \eps}{\alpha_{\bS}}
\]

\paragraph{Failure rate.}
Note that there are at most $\binom{|\bA|}{k}$ possible candidate sets in $\bC_k$ for each $k \in \{0, 1, \ldots, |\bA|\}$.
Since we invoked two calls to $\ApproxCondInd$ in iteration $k$, each with failure parameter $\delta w_k / 2$, union bound tells us that the probability of \emph{any} call failing across all calls is at most
\[
\sum_{k=0}^{|\bS'|} 2 \cdot \frac{\delta w_k}{2} \cdot \binom{|\bA|}{k}
= \sum_{k=0}^{|\bS'|} \delta \cdot \frac{1}{|\bA| \cdot \binom{|\bA|}{k}} \cdot \binom{|\bA|}{k}
= \sum_{k=0}^{|\bS'|} \frac{\delta}{|\bA|}
\leq \frac{\delta \cdot |\bS|}{|\bA|}
\leq \delta
\]

\paragraph{Sample complexity.}
Since we are using union bound to bound our overall failure probability, we can reuse samples in all our calls to \ApproxCondInd.
Thus, the total sample complexity is attributed to the final call when $k = |\bS'|$.
Such an invocation of $\ApproxCondInd$ uses $\wt{\cO} \left( \frac{1}{\eps^2} \cdot \sqrt{|\bSigma_{\bX}| \cdot |\bSigma_{\bY}| \cdot |\bSigma_{\bA \setminus \bS'}| \cdot |\bSigma_{\bS'}|} \cdot \log \frac{1}{\delta w_k} \right)$ samples according to \cref{lem:approxcondind-guarantees} and $w_k = \left( |\bA| \cdot \binom{|\bA|}{k} \right)^{-1}$, so the total number of samples used is at most
\[
\wt{\cO} \left( \frac{1}{\eps^2} \cdot \sqrt{|\bSigma_{\bX}| \cdot |\bSigma_{\bY}| \cdot |\bSigma_{\bA \setminus \bS'}| \cdot |\bSigma_{\bS'}|} \cdot \log \frac{1}{\delta w_k} \right)
\subseteq \wt{\cO} \left( \frac{|\bS'|}{\eps^2} \cdot \sqrt{|\bSigma_{\bX}| \cdot |\bSigma_{\bY}| \cdot |\bSigma_{\bA}|} \cdot \log \frac{1}{\delta} \right)
\]
We omit $\log |\bA|$ within $\wt{\cO}(\cdot)$ because $|\bA| \leq |\bSigma_{\bA}|$.
\end{proof}

\subsection{Deferred probabilistic manipulations}
\label{sec:appendix-probabilistic-manipulations}

In the proof of \cref{lem:minimal-soundness} and \cref{thm:BAMBA}, we skipped the full derivation of
\[
\sum_{\bs} \bbP(\bs) \cdot \bbP(\by \mid \bx, \bs) \cdot \sum_{\bs' \setminus \bs} \bbP(\bs' \setminus \bs \mid \by, \bx, \bs)
= \sum_{\bs \cup \bs'} \bbP(\bs \cup \bs') \cdot \bbP(\by \mid \bx, \bs \cup \bs') \cdot \frac{\bbP(\bx \mid \bs \cup \bs')}{\bbP(\bx \mid \bs)}
\]
It is obtained via a series of standard probabilistic manipulations:
\begin{align*}
&\; \sum_{\bs} \bbP(\bs) \cdot \bbP(\by \mid \bx, \bs) \cdot \sum_{\bs' \setminus \bs} \bbP(\bs' \setminus \bs \mid \by, \bx, \bs)\\
= &\; \sum_{\bs \cup \bs'} \bbP(\bs) \cdot \bbP(\by, \bs' \setminus \bs \mid \bx, \bs) \tag{Since $\bbP(\by \mid \bx, \bs) \cdot \bbP(\bs' \setminus \bs \mid \by, \bx, \bs) = \bbP(\by, \bs' \setminus \bs \mid \bx, \bs)$}\\
= &\; \sum_{\bs \cup \bs'} \frac{\bbP(\bs)}{\bbP(\bs \mid \bx)} \cdot \bbP(\by, \bs \cup \bs' \mid \bx) \tag{Since $\bbP(\by, \bs \cup \bs' \mid \bx) = \bbP(\bs \mid \bx) \cdot \bbP(\by, \bs' \setminus \bs \mid \bx, \bs)$}\\
= &\; \sum_{\bs \cup \bs'} \bbP(\by \mid \bx, \bs \cup \bs') \cdot \bbP(\bs \cup \bs' \mid \bx) \cdot \frac{\bbP(\bs)}{\bbP(\bs \mid \bx)} \tag{Since $\bbP(\by, \bs \cup \bs' \mid \bx) = \bbP(\bs \cup \bs' \mid \bx) \cdot \bbP(\by \mid \bx, \bs \cup \bs')$}\\
= &\; \sum_{\bs \cup \bs'} \bbP(\by \mid \bx, \bs \cup \bs') \cdot \bbP(\bs \cup \bs') \cdot \frac{\bbP(\bs) \cdot \bbP(\bx \mid \bs \cup \bs')}{\bbP(\bx) \cdot \bbP(\bs \mid \bx)} \tag{Since $\bbP(\bs \cup \bs' \mid \bx) = \frac{\bbP(\bx, \bs \cup \bs')}{\bbP(\bx)} = \frac{\bbP(\bs \cup \bs') \cdot \bbP(\bx \mid \bs \cup \bs')}{\bbP(\bx)}$}\\
= &\; \sum_{\bs \cup \bs'} \bbP(\by \mid \bx, \bs \cup \bs') \cdot \bbP(\bs \cup \bs') \cdot \frac{\bbP(\bx \mid \bs \cup \bs')}{\bbP(\bx \mid \bs)} \tag{Since $\bbP(\bx \mid \bs) = \frac{\bbP(\bx) \cdot \bbP(\bs \mid \bx)}{\bbP(\bs)}$}\\
= &\; \sum_{\bs \cup \bs'} \bbP(\bs \cup \bs') \cdot \bbP(\by \mid \bx, \bs \cup \bs') \cdot \frac{\bbP(\bx \mid \bs \cup \bs')}{\bbP(\bx \mid \bs)} \tag{Swap positions of $\bbP(\bs \cup \bs')$ and $\bbP(\by \mid \bx, \bs \cup \bs')$}
\end{align*}

In the proof of \cref{thm:BAMBA}, we also skipped the derivations of
\[
\sum_{\bs \cup \bs'} \bbP(\bx \mid \bs \cup \bs') \cdot \bbP(\bs \cup \bs') \cdot \left| \bbP(\by \mid \bx, \bs') - \bbP(\by \mid \bx, \bs \cup \bs') \right| \leq \eps
\quad
\text{and}
\quad
\sum_{\bs \cup \bs'} \bbP(\bs \cup \bs') \cdot \left| \bbP(\bx \mid \bs) - \bbP(\bx \mid \bs \cup \bs') \right| \leq \eps
\]
under the assumptions of $\Delta_{\bY \ind \bS \setminus \bS' \mid \bX \cup \bS'} \leq \eps$ and $\Delta_{\bX \ind \bS' \setminus \bS \mid \bS} \leq \eps$ respectively.
We derive them below:

\begin{align*}
&\; \sum_{\bs \cup \bs'} \bbP(\bx \mid \bs \cup \bs') \cdot \bbP(\bs \cup \bs') \cdot \left| \bbP(\by \mid \bx, \bs') - \bbP(\by \mid \bx, \bs \cup \bs') \right|\\
= &\; \sum_{\bs \cup \bs'} \bbP(\bx, \bs \cup \bs') \cdot \left| \bbP(\by \mid \bx, \bs') - \bbP(\by \mid \bx, \bs \cup \bs') \right| \tag{Since $\bX \cap (\bS \cup \bS') = \emptyset$}\\
= &\; \sum_{\bs \cup \bs'} \bbP(\bx, \bs') \cdot \bbP(\bs \setminus \bs' \mid \bx, \bs') \cdot \left| \bbP(\by \mid \bx, \bs') - \bbP(\by \mid \bx, \bs \cup \bs') \right|\\
= &\; \sum_{\bs \cup \bs'} \bbP(\bx, \bs') \cdot \left| \bbP(\by \mid \bx, \bs') \cdot \bbP(\bs \setminus \bs' \mid \bx, \bs') - \bbP(\by \cup (\bs \setminus \bs') \mid \bx, \bs') \right|\\
\leq &\; \sum_{\by, \bx, \bs \cup \bs'} \bbP(\bx, \bs') \cdot \left| \bbP(\by \mid \bx, \bs') \cdot \bbP(\bs \setminus \bs' \mid \bx, \bs') - \bbP(\by \cup (\bs \setminus \bs') \mid \bx, \bs') \right| \tag{Since we sum over all values of $\bSigma_{\bX}$ and $\bSigma_{\bY}$}\\
\leq &\; \eps \tag{when $\Delta_{\bY \ind \bS \setminus \bS' \mid \bX \cup \bS'} \leq \eps$}
\end{align*}

\begin{align*}
&\; \sum_{\bs \cup \bs'} \bbP(\bs \cup \bs') \cdot \left| \bbP(\bx \mid \bs) - \bbP(\bx \mid \bs \cup \bs') \right|\\
= &\; \sum_{\bs \cup \bs'} \bbP(\bs) \cdot \bbP(\bs' \setminus \bs \mid \bs) \cdot \left| \bbP(\bx \mid \bs) - \bbP(\bx \cup (\bs' \setminus \bs) \mid \bs) \right|\\
= &\; \sum_{\bs \cup \bs'} \bbP(\bs) \cdot \left| \bbP(\bx \mid \bs) \cdot \bbP(\bs' \setminus \bs \mid \bs) - \bbP(\bx \cup (\bs' \setminus \bs) \mid \bs) \right|\\
\leq &\; \sum_{\bx, \bs \cup \bs'} \bbP(\bs) \cdot \left| \bbP(\bx \mid \bs) \cdot \bbP(\bs' \setminus \bs \mid \bs) - \bbP(\bx \cup (\bs' \setminus \bs) \mid \bs) \right| \tag{Since we sum over all values of $\bSigma_{\bX}$}\\
\leq &\; \eps \tag{when $\Delta_{\bX \ind \bS' \setminus \bS \mid \bS} \leq \eps$}
\end{align*}

\subsection{Deferred derivations for hardness proof}
\label{sec:appendix-hardness}

In the proof of \cref{lem:misspecification-error-lower-bound}, we argued that the distribution $\bbP$ described in \cref{fig:hardness-alpha} has the following well-defined conditional probabilities:
\begin{center}
\begin{tabular}{cc|cccc}
\toprule
$a$ & $b$ & $\bbP(b \mid a)$ & $\bbP(X = 0 \mid a,b)$ & $\bbP(X = 0 \mid a)$ & $\sum_x | \bbP(x \mid a,b) - \bbP(x \mid a)|$\\
\midrule
$0$ & $0$ & $\sqrt{\eps}/2$ & $1 - \alpha + \sqrt{\eps}/2$ & $1 - \alpha + \eps/4$ & $\sqrt{\eps} - \eps/2$\\
$0$ & $1$ & $1-\sqrt{\eps}/2$ & $1 - \alpha$ & $1 - \alpha + \eps/4$ & $\eps/2$\\
$1$ & $0$ & $1-\sqrt{\eps}/2$ & $\alpha$ & $\alpha - \eps/4$ & $\eps/2$\\
$1$ & $1$ & $\sqrt{\eps}/2$ & $\alpha - \sqrt{\eps}/2$ & $\alpha - \eps/4$ & $\sqrt{\eps} - \eps/2$\\
\bottomrule
\end{tabular}
\end{center}

For convenience, we produce \cref{fig:hardness-alpha} below.

\begin{figure}[htb]
\centering
\resizebox{0.9\linewidth}{!}{
\begin{tikzpicture}
\node[] at (1,0.75) {$\cG$};
\node[draw, thick, circle] at (0,0) (A) {$A$};
\node[draw, thick, circle] at (2,0) (B) {$B$};
\node[draw, thick, circle] at (0,-2) (X) {$X$};
\node[draw, thick, circle] at (2,-2) (Y) {$Y$};
\draw[thick, -Stealth] (A) -- (B);
\draw[thick, -Stealth] (A) -- (X);
\draw[thick, -Stealth] (A) -- (Y);
\draw[thick, -Stealth] (B) -- (X);
\draw[thick, -Stealth] (B) -- (Y);
\draw[thick, -Stealth] (X) -- (Y);

\node[] at (6,-0.75) {\begin{tabular}{l}
$A = \begin{cases}
1 & \text{w.p.\ $\frac{\eps}{4 \alpha} \cdot \frac{\alpha - \eps/4}{1 - \sqrt{\eps}/2}$}\\
0 & \text{else}
\end{cases}$
\\
\\
$B = \begin{cases}
1 - A & \text{w.p.\ $1 - \sqrt{\eps}$}\\
0 & \text{w.p.\ $\sqrt{\eps}/2$}\\
1 & \text{w.p.\ $\sqrt{\eps}/2$}
\end{cases}$
\end{tabular}};

\node[] at (11,-0.75) {\begin{tabular}{l}
$X = \begin{cases}
A & \text{w.p.\ $1 - \alpha$}\\
1 - A & \text{w.p.\ $\alpha - \sqrt{\eps}/2$}\\
B & \text{w.p.\ $\sqrt{\eps}/2$}
\end{cases}$
\\
\\
$Y = \begin{cases}
1 & \text{if $X = 0, A = 1, B = 0$}\\
0 & \text{else}
\end{cases}$
\end{tabular}};
\end{tikzpicture}
}
\caption{Reproduced: Probability distribution $\bbP$ defined over 4 binary variables $\{A, B, X, Y\}$ in a topological ordering of $A \prec B \prec X \prec Y$ with parameters $\eps$ and $\alpha$, where $0 < \sqrt{\eps} \leq \alpha \leq 1/2$.}
\label{fig:hardness-alpha-reproduced}
\end{figure}
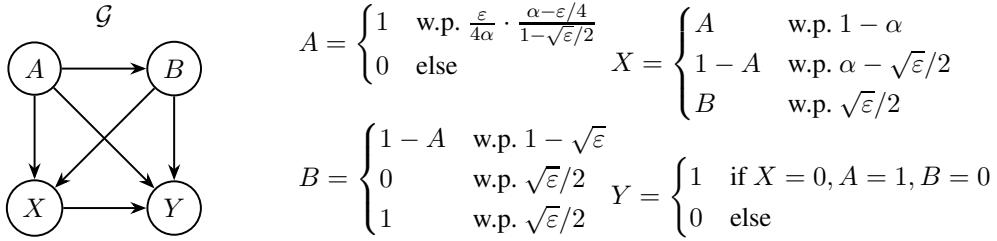

We first check that all the (conditional) probabilities of $\bbP$ are well-defined.
Since $0 < \sqrt{\eps} \leq \alpha \leq 1/2$, the only non-straightforward term to verify is $\bbP(A=1)$.
Observe that
\[
\frac{\eps}{4 \alpha} \cdot \frac{\alpha - \eps/4}{1 - \sqrt{\eps}/2} \leq 1
\iff
\eps \cdot (\alpha - \eps/4) \leq 4 \alpha \cdot (1 - \sqrt{\eps}/2)
\iff
2 \alpha \sqrt{\eps} + \alpha \eps - \eps^2 / 4 \leq 4 \alpha
\]
which is true as $0 < \eps < \sqrt{\eps} < \alpha \leq 1$ implies
$
2 \alpha \sqrt{\eps} +  \alpha \eps - \eps^2 / 4
\leq 3 \alpha \sqrt{\eps}
\leq 3 \alpha
\leq 4 \alpha
$.
Therefore, $0 \leq \bbP(A = 1) \leq 1$.

We now proceed to verify the conditional probabilities shown in the table above.
For instance,
\begin{align*}
&\; \bbP(X = 0 \mid A = 0)\\
= &\; \bbP(B = 0 \mid A = 0) \cdot \bbP(X = 0 \mid A = 0, B = 0) + \bbP(B = 1 \mid A = 0) \cdot \bbP(X = 0 \mid A = 0, B = 1)\\
= &\; (\sqrt{\eps}/2) \cdot (1 - \alpha + \sqrt{\eps}/2) + (1-\sqrt{\eps}/2) \cdot (1 - \alpha)\\
= &\; 1 - \alpha + \eps/4
\end{align*}
and
\begin{align*}
&\; \bbP(X = 0 \mid A = 1)\\
= &\; \bbP(B = 0 \mid A = 1) \cdot \bbP(X = 0 \mid A = 1, B = 0) + \bbP(B = 1 \mid A = 1) \cdot \bbP(X = 0 \mid A = 1, B = 1)\\
= &\; (1 - \sqrt{\eps}/2) \cdot \alpha + (\sqrt{\eps}/2) \cdot (\alpha - \sqrt{\eps}/2)\\
= &\; \alpha - \eps/4\\
= &\; 1 - \bbP(X = 0 \mid A = 0)
\end{align*}

The detailed workings for $\sum_x | \bbP(x \mid a,b) - \bbP(x \mid a)|$ for different values of $a,b \in \{0,1\}$ are as follows:

\textbf{When $A = 0$ and $B = 0$:}

\begin{align*}
&\; \sum_x | \bbP(x \mid a,b) - \bbP(x \mid a)|\\
= &\; \left| \bbP(X=0 \mid A=0, B=0) - \bbP(X=0 \mid A=0) \right| + \left| \bbP(X=1 \mid A=0, B=0) - \bbP(X=1 \mid A=0) \right|\\
= &\; \left| (1 - \alpha + \sqrt{\eps}/2) - (1 - \alpha + \eps/4) \right| + \left| (\alpha - \sqrt{\eps}/2) - (\alpha - \eps/4) \right|\\
= &\; 2 \left( \sqrt{\eps}/2 - \eps/4 \right)\\
= &\; \sqrt{\eps} - \eps/2
\end{align*}

\textbf{When $A = 0$ and $B = 1$:}

\begin{align*}
&\; \sum_x | \bbP(x \mid a,b) - \bbP(x \mid a)|\\
= &\; \left| \bbP(X=0 \mid A=0, B=1) - \bbP(X=0 \mid A=0) \right| + \left| \bbP(X=1 \mid A=0, B=1) - \bbP(X=1 \mid A=0) \right|\\
= &\; \left| (1 - \alpha) - (1 - \alpha + \eps/4) \right| + \left| (\alpha) - (\alpha - \eps/4) \right|\\
= &\; 2 \left( \eps/4 \right)\\
= &\; \eps/2
\end{align*}

\textbf{When $A = 1$ and $B = 0$:}

\begin{align*}
&\; \sum_x | \bbP(x \mid a,b) - \bbP(x \mid a)|\\
= &\; \left| \bbP(X=0 \mid A=1, B=0) - \bbP(X=0 \mid A=1) \right| + \left| \bbP(X=1 \mid A=1, B=0) - \bbP(X=1 \mid A=1) \right|\\
= &\; \left| (\alpha) - (\alpha - \eps/4) \right| + \left| (1 - \alpha) - (1 - \alpha + \eps/4) \right|\\
= &\; 2 \left( \eps/4 \right)\\
= &\; \eps/2
\end{align*}

\textbf{When $A = 1$ and $B = 1$:}

\begin{align*}
&\; \sum_x | \bbP(x \mid a,b) - \bbP(x \mid a)|\\
= &\; \left| \bbP(X=0 \mid A=1, B=1) - \bbP(X=0 \mid A=1) \right| + \left| \bbP(X=1 \mid A=1, B=1) - \bbP(X=1 \mid A=1) \right|\\
= &\; \left| (\alpha - \sqrt{\eps}/2) - (\alpha - \eps/4) \right| + \left| (1 - \alpha + \sqrt{\eps}/2) - (1 - \alpha + \eps/4) \right|\\
= &\; 2 \left( \sqrt{\eps}/2 - \eps/4 \right)\\
= &\; \sqrt{\eps} - \eps/2
\end{align*}

\subsection{Estimating causal effects using AMBA and BAMBA}
\label{sec:appendix-putting-together}

If we re-express the results of \cref{thm:estimation-error}, \cref{thm:AMBA} and \cref{thm:BAMBA} in terms of an upper bound on error for a fixed number of samples $n$, we get the following three corollaries.

\begin{mycorollary}[Estimation corollary]
\label{cor:estimation-error}
Suppose we are given (1) failure tolerance $\delta > 0$, (2) $n$ i.i.d.\ samples from distribution $\bbP(\bV)$, (3) a subset $\bA \subseteq \bV$ with $\alpha_{\bA} = \max_{\ba \in \bSigma_{\bA}} \bbP(\bx \mid \ba)$.
Then, there is an algorithm that produces an estimate $\hatT_{\bA, \bx, \by}$ such that $\Pr(|\hatT_{\bA, \bx, \by} - T_{\bA, \bx, \by}| \leq \eps) \geq 1 - \delta$ for some error term
\[
\eps \in \wt{\cO} \left( \frac{|\bSigma_{\bA}|}{n \alpha_{\bA}} + \frac{1}{\sqrt{n \alpha_{\bA}}} + \sqrt{\frac{|\bSigma_{\bA}|}{n}} \right)
\]
\end{mycorollary}
\begin{proof}
From \cref{thm:estimation-error}, we know that $\wt{\cO} \left( \left( \frac{|\bSigma_{\bA}|}{\eps \alpha_{\bA}} + \frac{1}{\eps^2 \alpha_{\bA}} + \frac{|\bSigma_{\bA}|}{\eps^2} \right) \cdot \log \left( \frac{1}{\delta} \right) \right)$ samples suffice to produce an estimate $\hatT_{\bA, \bx, \by}$ such that $\Pr(|\hatT_{\bA, \bx, \by} - T_{\bA, \bx, \by}| \leq \eps) \geq 1 - \delta$.
Ignoring the logarithmic terms and constant factors, the result follows by re-expressing $n \leq \frac{|\bSigma_{\bA}|}{\eps \alpha_{\bA}} + \frac{1}{\eps^2 \alpha_{\bA}} + \frac{|\bSigma_{\bA}|}{\eps^2}$ in terms of $\eps$.
\end{proof}

\begin{mycorollary}[\AMBAshort\ corollary]
\label{cor:AMBA}
Suppose we are given (1) failure tolerance $\delta > 0$, (2) $n$ i.i.d.\ samples from distribution $\bbP(\bV)$, and (3) an arbitrary subset $\bA \subseteq \VminusXY$.
Then, there is an algorithm that produces a subset $\bS \subseteq \bA$ such that $\Pr \left( \Delta_{\bX \ind \bA \setminus \bS \mid \bS} > \eps \right) \geq 1 - \delta$ and $\Pr \left( | T_{\bS, \bx, \by} - T_{\bA, \bx, \by} | \leq \eps \right) \geq 1 - \delta$ for some error term
\[
\eps \in \wt{\cO} \left( \frac{1}{\alpha_{\bS}} \cdot \sqrt{\frac{|\bS|}{n}} \cdot \left( \left| \bSigma_{\bX} \right| \cdot \left| \bSigma_{\bA} \right| \right)^{\frac{1}{4}} \right)
\]
\end{mycorollary}
\begin{proof}
From \cref{thm:AMBA}, we know that $\wt{\cO} \left( \frac{|\bS|}{\eps^2} \cdot \sqrt{|\bSigma_{\bX}| \cdot |\bSigma_{\bA}|} \cdot \log \frac{1}{\delta} \right)$ samples suffice to produce a subset $\bS \subseteq \bA$ such that $\Pr \left( \Delta_{\bX \ind \bA \setminus \bS \mid \bS} > \eps \right) \geq 1 - \delta$ and $\Pr \left( | T_{\bS, \bx, \by} - T_{\bA, \bx, \by} | \leq \frac{\eps}{\alpha_{\bS}} \right) \geq 1 - \delta$.
Ignoring the logarithmic terms and constant factors, the result follows by re-expressing $n = \frac{|\bS|}{(\eps')^2} \cdot \sqrt{|\bSigma_{\bX}| \cdot |\bSigma_{\bA}|}$ in terms of $\eps' = \eps \alpha_{\bS} \leq \eps$.
\end{proof}

\begin{mycorollary}[\BAMBAshort\ corollary]
\label{cor:BAMBA}
Suppose we are given (1) failure tolerance $\delta > 0$, (2) $n$ i.i.d.\ samples from distribution $\bbP(\bV)$, (3) an arbitrary subset $\bA \subseteq \VminusXY$, and (4) an $\eps$-Markov blanket $\bS \subseteq \bA$.
Then, there is an algorithm that produces a subset $\bS' \subseteq \bA$ such that $|\bSigma_{\bS'}| \leq |\bSigma_{\bS}|$ and $\Pr \left( | T_{\bS', \bx, \by} - T_{\bA, \bx, \by} | \leq \eps \right) \geq 1 - \delta$ for some error term
\[
\eps \in \wt{\cO} \left( \frac{1}{\alpha_{\bS}} \cdot \sqrt{\frac{|\bS'|}{n}} \cdot \left( \left| \bSigma_{\bX} \right| \cdot \left| \bSigma_{\bY} \right| \cdot \left| \bSigma_{\bA} \right| \right)^{\frac{1}{4}} \right)
\]
\end{mycorollary}
\begin{proof}
From \cref{thm:BAMBA}, we know that $\wt{\cO} \left( \frac{|\bS'|}{\eps^2} \cdot \sqrt{|\bSigma_{\bX}| \cdot |\bSigma_{\bY}| \cdot |\bSigma_{\bA}|} \cdot \log \frac{1}{\delta} \right)$ samples suffice to produce a subset $\bS' \subseteq \bA$ such that $|\bSigma_{\bS'}| \leq |\bSigma_{\bS}|$ and $\Pr \left( | T_{\bS', \bx, \by} - T_{\bA, \bx, \by} | \leq \frac{\eps}{\alpha_{\bS}} \right) \geq 1 - \delta$.
Ignoring the logarithmic terms and constant factors, the result follows by re-expressing $n = \frac{|\bS'|}{(\eps')^2} \cdot \sqrt{|\bSigma_{\bX}| \cdot |\bSigma_{\bY}| \cdot |\bSigma_{\bA}|}$ in terms of $\eps' = \eps \alpha_{\bS} \leq \eps$.
\end{proof}

In light of \cref{cor:estimation-error}, \cref{cor:AMBA}, and \cref{cor:BAMBA}, there are a couple of ways one could attempt to estimate $\bbP_{\bx}(\by)$ when given a valid adjustment set $\bZ \subseteq \VminusXY$:
\begin{enumerate}
    \item Directly estimate using $\bZ$.
    By \cref{cor:estimation-error}, this yields an error of 
    \[
    |\bbP_{\bx}(\by) - \wh{\bbP}_{\bx}(\by)|
    = |T_{\bZ, \bx, \by} - \wh{T}_{\bZ, \bx, \by}|
    \in \wt{\cO} \left( \frac{|\bSigma_{\bZ}|}{n \alpha_{\bZ}} + \frac{1}{\sqrt{n \alpha_{\bZ}}} + \sqrt{\frac{|\bSigma_{\bZ}|}{n}} \right)
    \]
    \item Use \AMBAshort\ on $\bZ$ to produce a subset $\bS \subseteq \bZ$ and estimate using $\bS$.
    By \cref{cor:estimation-error} and \cref{cor:AMBA}, this yields an error of
    \begin{multline*}
    |\bbP_{\bx}(\by) - \wh{\bbP}_{\bx}(\by)|
    = |T_{\bZ, \bx, \by} - \hatT_{\bS, \bx, \by}|
    \leq |T_{\bZ, \bx, \by} - T_{\bS, \bx, \by}| + |T_{\bS, \bx, \by} - \hatT_{\bS, \bx, \by}|\\
    \in \wt{\cO} \left( \frac{1}{\alpha_{\bS}} \cdot \sqrt{\frac{|\bS|}{n}} \cdot \left( \left| \bSigma_{\bX} \right| \cdot \left| \bSigma_{\bZ} \right| \right)^{\frac{1}{4}}
    + \frac{|\bSigma_{\bS}|}{n \alpha_{\bS}} + \frac{1}{\sqrt{n \alpha_{\bS}}} + \sqrt{\frac{|\bSigma_{\bS}|}{n}} \right)
    \end{multline*}
    \item Use \AMBAshort\ on $\bZ$ to produce a subset $\bS \subseteq \bZ$, then use \BAMBAshort\ to further produce subset $\bS'$, and then estimate using $\bS'$.
    By \cref{cor:estimation-error}, \cref{cor:AMBA}, and \cref{cor:BAMBA}, this yields an error of
    \begin{multline*}
    |\bbP_{\bx}(\by) - \wh{\bbP}_{\bx}(\by)|
    = |T_{\bZ, \bx, \by} - \hatT_{\bS, \bx, \by}|
    \leq |T_{\bZ, \bx, \by} - T_{\bS', \bx, \by}| + |T_{\bS', \bx, \by} - \hatT_{\bS', \bx, \by}|\\
    \in \wt{\cO} \left( \frac{1}{\alpha_{\bS}} \cdot \sqrt{\frac{|\bS|}{n}} \cdot \left( \left| \bSigma_{\bX} \right| \cdot \left| \bSigma_{\bZ} \right| \right)^{\frac{1}{4}}
    + \frac{1}{\alpha_{\bS}} \cdot \sqrt{\frac{|\bS'|}{n}} \cdot \left( \left| \bSigma_{\bX} \right| \cdot \left| \bSigma_{\bY} \right| \cdot \left| \bSigma_{\bZ} \right| \right)^{\frac{1}{4}}
    + \frac{|\bSigma_{\bS'}|}{n \alpha_{\bS'}} + \frac{1}{\sqrt{n \alpha_{\bS'}}} + \sqrt{\frac{|\bSigma_{\bS'}|}{n}} \right)
    \end{multline*}
\end{enumerate}

In any cases 2 and 3, with appropriate constant factors, we see that
\[
|\bbP_{\bx}(\by) - \wh{\bbP}_{\bx}(\by)|
= |T_{\bZ, \bx, \by} - \hatT_{\bS, \bx, \by}|
\leq |T_{\bZ, \bx, \by} - T_{\bS, \bx, \by}| + |T_{\bS, \bx, \by} - \hatT_{\bS, \bx, \by}|
\leq \eps + \eps
= 2 \eps
\]
The following lemma tells us that $\alpha_{\bZ} \leq \alpha_{\bS}$ and $\alpha_{\bZ} \leq \alpha_{\bS'}$, i.e.\ $\frac{1}{\alpha_{\bS}} \leq \frac{1}{\alpha_{\bZ}}$ and $\frac{1}{\alpha_{\bS'}} \leq \frac{1}{\alpha_{\bZ}}$, so it is always beneficial to use a smaller subset with respect to the error incurred by estimation in \cref{cor:estimation-error}.

\begin{mylemma}
\label{lem:superset-has-smaller-alpha}
For any value $\bx$ for $\bX$ and subsets $\bA \subseteq \bB \subseteq \bV \setminus \bX$, we have
\[
\alpha_{\bA}
= \min_{\ba} \bbP(\bx \mid \ba)
\geq \min_{\bb} \bbP(\bx \mid \bb)
= \alpha_{\bB}
\]
\end{mylemma}
\begin{proof}
Fix an arbitrary values of $\bx$ for $\bX$ and $\ba$ for $\bA$, we see that
\[
\bbP(\bx \mid \ba)
= \sum_{\bb \setminus \ba} \bbP(\bx, \bb \setminus \ba \mid \ba)
\geq \min_{\bb} \bbP(\bx \mid \bb) \cdot \sum_{\bb \setminus \ba} \bbP(\bb \setminus \ba \mid \ba)
= \min_{\bb} \bbP(\bx \mid \bb)
\]
Therefore, $\min_{\ba} \bbP(\bx \mid \ba) \geq \min_{\bb} \bbP(\bx \mid \bb)$.
\end{proof}

Observe that $\frac{1}{\alpha_{\bS}} \leq \frac{1}{\alpha_{\bZ}}$ from \cref{lem:superset-has-smaller-alpha} and $|\bSigma_{\bS}| \leq |\bSigma_{\bZ}|$ since $\bS \subseteq \bZ$.
So, the second approach of estimating $\bbP_{\bx}(\by)$ using the subset $\bS \subseteq \bZ$ produced by \AMBAshort\ would yield an asymptotically smaller error than directly using $\bZ$ whenever $\frac{1}{\alpha_{\bS}} \cdot \sqrt{\frac{|\bS|}{n}} \cdot \left( \left| \bSigma_{\bX} \right| \cdot \left| \bSigma_{\bZ} \right| \right)^{\frac{1}{4}} \leq \frac{|\bSigma_{\bZ}|}{n \alpha_{\bS}} + \frac{1}{\sqrt{n \alpha_{\bS}}} + \sqrt{\frac{|\bSigma_{\bZ}|}{n}}$.
This happens when
\begin{equation}
\label{eq:AMBA-decision}
|\bS| \cdot \sqrt{\frac{|\bSigma_{\bX}|}{|\bSigma_{\bZ}|}}
< \max \left\{
\frac{|\bSigma_{\bZ}|}{n},
\;
\frac{\alpha_{\bS}}{|\bSigma_{\bZ}|},
\;
\alpha_{\bS}^2
\right\}
\end{equation}
Observe that we know all terms in \cref{eq:AMBA-decision} except for $\alpha_{\bS}$.
For small $n$, say when $n \ll |\bSigma_{\bZ}|$, the first term justifies estimating using the subset $\bS$ produced by \AMBAshort\ instead of directly estimating using $\bZ$.
However, for large $n$, one would need to make the decision based on $\alpha_{\bS}$.
A similar kind of decision has to be made whether the third approach, of running \AMBAshort\ to produce $\bS \subseteq \bZ$ then \BAMBAshort\ to produce $\bS' \subseteq \bZ$, would yield a smaller estimation error.
Note that $|\bSigma_{\bS'}| \leq |\bSigma_{\bS}|$ would imply $|\bS'| \leq |\bS|$ when all variables have the same domain size.


\paccausaleffectestimationspecialcase*
\begin{proof}
Consider the following algorithm:
\begin{enumerate}
    \item Run \AMBAshort\ to obtain $\bS \subseteq \bZ$
    \item Check if $|\bS| \cdot \sqrt{\frac{|\bSigma_{\bX}|}{|\bSigma_{\bZ}|}} < \max \left\{ \frac{|\bSigma_{\bZ}|}{n}, \frac{\alpha_{\bS}}{|\bSigma_{\bZ}|}, \alpha_{\bS}^2 \right\}$ according to \cref{eq:AMBA-decision}
    \item If so, run \BAMBAshort\ to obtain $\bS' \subseteq \bZ$ and produce estimate $\hatbbP_{\bx}(\by) = \hatT_{\bS',\bx,\by}$
    \item Otherwise, produce estimate $\hatbbP_{\bx}(\by) = \hatT_{\bZ,\bx,\by}$
\end{enumerate}
That is, depending on \cref{eq:AMBA-decision}, we decide to perform estimation based on $\bS^* = \bS'$ or $\bS^* = \bZ$.
It remains to show that the bound holds for each case separately while noting that $\alpha_{\bS}, \alpha_{\bS'}, \alpha_{\bZ} \geq \alpha$.

\textbf{Case 1}: $|\bS| \cdot \sqrt{\frac{|\bSigma_{\bX}|}{|\bSigma_{\bZ}|}} < \max \left\{ \frac{|\bSigma_{\bZ}|}{n}, \frac{\alpha_{\bS}}{|\bSigma_{\bZ}|}, \alpha_{\bS}^2 \right\}$, so we estimate using $\bS^* = \bS'$ produced from \BAMBAshort\

This incurs an error of
\begin{align*}
&\; |\bbP_{\bx}(\by) - \wh{\bbP}_{\bx}(\by)|
= |T_{\bZ, \bx, \by} - \hatT_{\bS, \bx, \by}|
\leq |T_{\bZ, \bx, \by} - T_{\bS', \bx, \by}| + |T_{\bS', \bx, \by} - \hatT_{\bS', \bx, \by}|\\
\in &\; \wt{\cO} \left( \frac{1}{\alpha} \cdot \sqrt{\frac{|\bS|}{n}} \cdot \left( \left| \bSigma_{\bX} \right| \cdot \left| \bSigma_{\bZ} \right| \right)^{\frac{1}{4}}
+ \frac{1}{\alpha} \cdot \sqrt{\frac{|\bS'|}{n}} \cdot \left( \left| \bSigma_{\bX} \right| \cdot \left| \bSigma_{\bY} \right| \cdot \left| \bSigma_{\bZ} \right| \right)^{\frac{1}{4}}
+ \frac{|\bSigma_{\bS'}|}{n \alpha} + \frac{1}{\sqrt{n \alpha}} + \sqrt{\frac{|\bSigma_{\bS'}|}{n}} \right) \tag{From \cref{cor:estimation-error}, \cref{cor:AMBA}, and \cref{cor:BAMBA}}\\
\subseteq &\; \wt{\cO} \left( \frac{1}{\alpha} \cdot \sqrt{\frac{|\bZ|}{n}} \cdot \left( \left| \bSigma_{\bX} \right| \cdot \left| \bSigma_{\bY} \right| \cdot \left| \bSigma_{\bZ} \right| \right)^{\frac{1}{4}}
+ \frac{|\bSigma_{\bS'}|}{n \alpha} + \frac{1}{\sqrt{n \alpha}} + \sqrt{\frac{|\bSigma_{\bS'}|}{n}} \right) \tag{Since $\max\{|\bS|, |\bS'|\} \leq |\bZ|$}\\
\subseteq &\; \wt{\cO} \left(
\frac{1}{n} \cdot \frac{|\bSigma_{\bS^*}|}{\alpha} + \frac{1}{\sqrt{n}} \cdot \left( \frac{\sqrt{|\bZ|} \cdot \left( |\bSigma_{\bX}| \cdot |\bSigma_{\bY}| \cdot |\bSigma_{\bZ}| \right)^{\frac{1}{4}}}{\alpha} + \frac{1}{\sqrt{\alpha}} + \sqrt{|\bSigma_{\bS^*}|}
\right) \right) \tag{Since $\bS^* = \bS'$}
\end{align*}

\textbf{Case 2}: $|\bS| \cdot \sqrt{\frac{|\bSigma_{\bX}|}{|\bSigma_{\bZ}|}} \geq \max \left\{ \frac{|\bSigma_{\bZ}|}{n}, \frac{\alpha_{\bS}}{|\bSigma_{\bZ}|}, \alpha_{\bS}^2 \right\}$, so we estimate using $\bS^* = \bZ$

This incurs an error of
\begin{align*}
&\; |\bbP_{\bx}(\by) - \wh{\bbP}_{\bx}(\by)|
= |T_{\bZ, \bx, \by} - \wh{T}_{\bZ, \bx, \by}|\\
\in &\; \wt{\cO} \left( \frac{|\bSigma_{\bZ}|}{n \alpha} + \frac{1}{\sqrt{n \alpha}} + \sqrt{\frac{|\bSigma_{\bZ}|}{n}} \right) \tag{From \cref{cor:estimation-error}}\\
\subseteq &\; \wt{\cO} \left( \frac{|\bSigma_{\bS^*}|}{n \alpha} + \frac{1}{\sqrt{n \alpha}} + \sqrt{\frac{|\bSigma_{\bS^*}|}{n}} \right) \tag{Since $\bS^* = \bS'$}\\
\subseteq &\; \wt{\cO} \left(
\frac{1}{n} \cdot \frac{|\bSigma_{\bS^*}|}{\alpha} + \frac{1}{\sqrt{n}} \cdot \left( \frac{\sqrt{|\bZ|} \cdot \left( |\bSigma_{\bX}| \cdot |\bSigma_{\bY}| \cdot |\bSigma_{\bZ}| \right)^{\frac{1}{4}}}{\alpha} + \frac{1}{\sqrt{\alpha}} + \sqrt{|\bSigma_{\bS^*}|}
\right) \right) \tag{Adding more terms}
\end{align*}

Therefore, we see that the error upper bound holds for either case.
\end{proof}

\end{document}